\documentclass{emsprocart}

 \usepackage{tikz}
\usetikzlibrary{matrix,arrows,calc,snakes,patterns,decorations.markings}
\usepackage[mathscr]{eucal}
\usepackage{amscd}
\usepackage{amsfonts}
\usepackage{amsmath, amssymb, bbm, bm}
\usepackage{stmaryrd}
\usepackage{lscape}
\usepackage{url}

\usepackage{latexsym}
\usepackage{graphics}

\theoremstyle{plain}
\newtheorem{thm}{Theorem}[section]
\newtheorem{cor}[thm]{Corollary}
\newtheorem{lem}[thm]{Lemma}
\newtheorem{prop}[thm]{Proposition}

\theoremstyle{definition}

\newtheorem{defn}[thm]{Definition}

\newtheorem{ex}[thm]{Example}

\theoremstyle{remark}
\newtheorem{rem}[thm]{Remark}
\theoremstyle{setup}
\newtheorem{Setup}[thm]{Setup}
\newtheorem*{Key Statement}{Key Statement}
\newtheorem{Conjecture}[thm]{Conjecture}
\newtheorem{Question}[thm]{Question}
\numberwithin{equation}{section}

\newcommand{\ul}[1]{\underline{#1}}

\DeclareMathAlphabet{\mathpzc}{OT1}{pzc}{m}{it}

\renewcommand{\ker}{\mathsf{ker}}
\newcommand{\im}{\mathsf{im}}

\renewcommand{\dim}{\mathsf{dim}}

\DeclareMathOperator{\coh}{\mathsf{coh}}

\DeclareMathOperator{\Rep}{\mathsf{Rep}}

\DeclareMathOperator{\gldim}{\mathsf{gl.dim}}

\DeclareMathOperator{\Mod}{\mathsf{Mod}}

\DeclareMathOperator{\Hom}{\mathsf{Hom}}
\DeclareMathOperator{\RHom}{\mathsf{RHom}}

\DeclareMathOperator{\Ext}{\mathsf{Ext}}

\DeclareMathOperator{\End}{\mathsf{End}}

\newcommand{\opname}[1]{\operatorname{\mathsf{#1}}}

\newcommand{\ten}{\otimes}
\newcommand{\lten}{\overset{\opname{L}}{\ten}}

\renewcommand{\ker}{\opname{ker}\nolimits}

\newcommand{\thick}{\opname{thick}\nolimits}

\newcommand{\per}{\opname{per}\nolimits}

\renewcommand{\mod}{\opname{mod}\nolimits}

\newcommand{\proj}{\opname{proj}\nolimits}

\renewcommand{\Mod}{\opname{Mod}\nolimits}

\newcommand{\ca}{{\mathcal A}}

\newcommand{\cc}{{\mathcal C}}
\newcommand{\cd}{{\mathcal D}}

\newcommand{\cn}{{\mathcal N}}

\newcommand{\cq}{{\mathcal Q}}
\newcommand{\cs}{{\mathcal S}}
\newcommand{\ct}{{\mathcal T}}

\newcommand{\ra}{\rightarrow}

\newcommand{\id}{\mathbf{1}}

\input xy
\xyoption{all}

\newcommand{\bsm}{\begin{smallmatrix}}
\newcommand{\esm}{\end{smallmatrix}}

\renewcommand{\mod}{\mathsf{mod}}

\renewcommand{\AA}{\mathbb{A}}

\newcommand{\ZZ}{\mathbb{Z}}

\newcommand{\recollto}{%
        \mathrel{\vcenter{\mathsurround0pt
                \ialign{##\crcr
                        \noalign{\nointerlineskip}$\,\,\scriptstyle\leftarrow$\crcr
                        \noalign{\nointerlineskip}$\longrightarrow$\crcr
                        \noalign{\nointerlineskip}$\,\,\scriptstyle\leftarrow$\crcr
                }%
        }}%
}

\newcommand{\recoll}[3]{#1 \recollto #2 \recollto #3}                                                         

\contact[m.kalck@ed.ac.uk]{Martin Kalck, School of Mathematics, The University of Edinburgh, Peter Guthrie Tait Road, Edinburgh, EH9 3FD, Scotland, UK.}

\title[Derived categories of quasi-hereditary algebras and their derived composition series]{Derived categories of quasi-hereditary algebras \\ and their derived composition series}

\author[Martin Kalck]{Martin Kalck\thanks{This work was supported by the EPSRC Postdoctoral Fellowship EP/L017962/1.}}

\begin{document}

\begin{abstract}
We study composition series of derived module categories in the sense of Angeleri H{\"u}gel, K{\"o}nig \&  Liu for quasi-hereditary algebras. More precisely, we show that having a composition series with all factors being derived categories of vector spaces does not characterise derived categories of quasi-hereditay algebras. This gives a negative answer to a question of Liu \& Yang and the proof also confirms part of a conjecture of Bobi{\'n}ski \& Malicki. In another direction, we show that derived categories of quasi-hereditary algebras can have composition series with lots of different lengths and composition factors. In other words, there is no Jordan-H{\"o}lder property for composition series of derived categories of quasi-hereditary algebras.
\end{abstract}



\maketitle

\tableofcontents
\section{Introduction}
Triangulated categories are used and studied in different areas of mathematics and theoretical physics -- algebraic geometry (for example, with applications to classical problems in birational geometry, see e.g.  \cite{BKR, HomMMP}), representation theory (with relations to cluster algebras, starting with \cite{BMRRT} and perverse sheaves \cite{BBD} used in the proof of the Kazhdan-Lusztig conjectures), algebraic topology, string theory (via Kontsevich's Homological Mirror Symmetry conjecture \cite{Kontsevich}), ... . In general, triangulated categories are rather complicated structures and therefore techniques allowing a decomposition into more accessible pieces are important. In this article, we will focus on decompositions of triangulated categories $\ct$ of the form
\begin{align} \label{E:Recoll}Ê\recoll{\ct'}{\ct}{\ct''}\end{align} 
 called \emph{recollements}, which have properties similar to short exact sequences, see the discussion below. We refer to \cite{BBD} and Section \ref{S:RecollementsETC} for the precise definition.
 
 Quasi-hereditary algebras form an important class of finite dimensional algebras with relations to Lie theory (in fact this was the original motivation \cite{Scott}) and also exceptional sequences in algebraic geometry (see e.g. \cite{HillePerling} and \cite{BuchweitzLeuschkeVdB}). The category of finitely generated modules over a quasi-hereditary algebra is an example of a highest weight category and conversely, every highest weight category with finitely many simple objects is of this form \cite[Theorem 3.6]{CPS2}. Highest weight categories are also discussed in Krause's article in this volume \cite{Krause}: in particular, the category of strict polynomial functors admits the structure of a highest weight category. Moreover, work of Dlab \& Ringel \cite{DR2}Ê shows that every finite dimensional algebra admits a `resolution' by a quasi-hereditary algebra. A generalisation of this result led to Iyama's proof of the finiteness of Auslander's representation dimension \cite{Iyama}. Examples of quasi-hereditary algebras include blocks of category $\mathcal{O}$ and Schur algebras (see e.g. the articles by Krause \& K{\"u}lshammer in this volume \cite{Krause, Kuelshammer}). It is well-known that quasi-hereditary algebras may be defined in terms of sequences of recollements of abelian categories, see \cite{ParshallScott} and also \cite{Krause}.
    
 Recollements of derived categories induce long exact sequences in K-groups (see e.g. \cite{NeemanYao}), Hochschild homology and cyclic homology, see \cite[Remark 3.2]{KellerCyclic} for the latter two. They also give rise to long exact sequences involving Hochschild cohomology groups of all the algebras in the recollement, see e.g. \cite[Corollary 3]{HanHoch}. Moreover, recollements of derived module categories allow to reduce the proof of classical homological conjectures to simpler and smaller algebras, see e.g. \cite{Happel93}. Also $t$-structures on $\ct'$ and $\ct''$ can be glued to a $t$-structure on $\ct$ -- in fact, this was one of the main motivations in \cite{BBD}, where recollements arising from stratifications of topological spaces where used to construct so called \emph{perverse $t$-structures} giving rise to \emph{perverse sheaves}. Summing up, recollements behave similar as short exact sequences. 

This view was the starting point for a recent series of articles by Angeleri H{\"u}gel, K{\"o}nig, Liu (and Yang), see \cite{AKL, AKLJH, AKLY}, where one can find some historical background and also Yang's ICRA talk \cite{Yang}. We give a brief account of part of this work here.

In the presence of a notion of short exact sequence one can define simple objects as those which do not appear as middle terms of non-trivial short exact sequences. We call these objects \emph{triangulated simple}.
In analogy with short exact sequences of modules over rings, iteration (i.e. taking recollements of the outer terms $\ct'$ and $\ct''$ in \eqref{E:Recoll} and recollements of their outer terms and so forth until we reach \emph{triangulated simple} categories) leads to the notion of  \emph{(triangulated) composition series} (sometimes also called \emph{stratification} of triangulated categories). We call the triangulated simple categories appearing in this process \emph{triangulated composition factors}. Angeleri H{\"u}gel, K{\"o}nig \& Liu \cite[Example 6.1]{AKL} (see also Remark \ref{R:Kuznetsov} (ii)) show that the derived category $\cd=\cd(\Rep \mathcal{K})$ of representations of the Kronecker quiver $\mathcal{K}$ admits a triangulated composition series of infinite length, where all but one composition factor are not derived module categories. It is well-known that $\cd$ also has a composition series of length $2$ with factors derived categories of vector spaces. Because of this observation and `a lack of techniques to study the general triangulated categories appearing in triangulated composition series', they decided to restrict to triangulated composition series involving only derived module categories. We call these composition series \emph{derived composition series} and the corresponding simple factors \emph{derived simple}. We can now state the main questions of this article.

The first question asks for an analogue of the Jordan-H{\"o}lder property.

\begin{Question}[Jordan-H{\"o}lder] \label{Q:JordanHoelder}
When are derived composition series finite?
When are the derived composition factors unique up to reordering and equivalences? 
\end{Question}

\begin{Question}\label{Q:CharQH}
What does the existence of `special' derived composition series 
tell us about the original category?

In particular, Liu \& Yang \cite[Question 1.1]{derivedsimple} ask whether derived categories of quasi-hereditary algebras may be characterised as those derived categories admitting a derived composition series with all factors derived categories of vector spaces?
\end{Question}

\subsection*{On Question \ref{Q:JordanHoelder}}
Question \ref{Q:JordanHoelder} is known to have a positive answer for piecewise hereditary algebras \cite{AKLJH}, blocks of finite group algebras  \cite{LYBlocks}, Vossieck's derived discrete algebras \cite{Qin}, finite dimensional algebras with at most two simples \cite{derivedsimple}, commutative algebras and semisimple algebras.

It is known that the answer is negative in general, for non-uniqueness see \cite{ChenXi} (this involves non-artinian rings) and \cite{AKLY} (for counterexamples involving algebras of infinite global dimension). Moreover, it is not too surprising that the length of a derived composition series need not be finite, see e.g. \cite[Example 6.2]{AKL}. Question \ref{Q:JordanHoelder} was open for quasi-hereditary algebras and due to the existence of full exceptional sequences with good properties (given for example by the standard modules) it seems that there was some hope for a positive answer in this case. Building on work of Liu \& Yang \cite{derivedsimple}, we show that the derived category of the algebra $A'_2:=kQ_2/I'_2$ given by
\begin{align}
\begin{split}
\label{E:CJHalgebra}
\begin{tikzpicture}[description/.style={fill=white,inner sep=2pt}]
    \matrix (n) [matrix of math nodes, row sep=1em, ampersand replacement=\&,
                 column sep=1.3em, text height=1.5ex, text depth=0.25ex,
                 inner sep=0pt, nodes={inner xsep=0.3333em, inner
ysep=0.3333em}]
    {  \&\& 2 \\
       Q_2:= \& 3 \& \& 1   \\ 
    };
\path[-> ] (n-2-2) edge [bend left=20] node[ scale=0.75, fill=white] [midway] {$\gamma$} 
(n-2-4);
\path[-> ] (n-2-2) edge [bend right=20] node[ scale=0.75, fill=white] [midway] {$\beta$} 
(n-2-4);
\path[-> ] (n-2-4) edge node[ scale=0.75, xshift=9pt, yshift=3pt] [midway] {$\alpha_2$} 
(n-1-3);
\path[-> ] (n-1-3) edge node[ scale=0.75, xshift=-9pt, yshift=3pt] [midway] {$\alpha_1$} 
(n-2-2);
\draw (n-2-1) node[yshift=-29pt, xshift=43pt] {$I'_{2}:=( \alpha_{2} \beta, \gamma \alpha_{1})$};
\end{tikzpicture}
\end{split}
\end{align}
has at least two derived composition series of different length. Since $A'_2$ has global dimension $2$ it is quasi-hereditary \cite[Theorem 2]{DR}. Moreover, since $A'_2$ is gentle (see Definition \ref{D:Gentle}) its derived category is of tame representation type -- indeed, in this case the repetitive algebra is special biserial \cite{Ringel, Schroer}. Finally, all algebras appearing in our derived composition series have finite global dimension and all involved recollements are induced by idempotents. In particular, compared to the examples in \cite{ChenXi} and \cite{AKLY}, our examples have a quite different flavour.

We use this example as a starting point to construct quasi-hereditary algebras with an arbitrary number of composition series of different length (see Proposition \ref{P:MainJH}).

\begin{rem}
\begin{itemize}
\item[(i)] It follows from work of Dlab \& Ringel \cite{DR2} that every finite dimensional algebra $B$ can be written as $B=eAe$ where $A$ is quasi-hereditary and $e \in A$ is an idempotent. From this perspective the failure of the derived JH-property does not seem to be too surprising -- indeed in our example all recollements are induced by idempotents.
\item[(ii)] From a `practical' perspective the failure of the derived JH-property for quasi-hereditary algebras is not a problem -- it is well-known (see Lemma \ref{C:DerCompSeriesQH}) that there always exists a derived composition series with factors $\cd^b(k)$, which is convenient for inductively computing invariants. Continuing this line of thought JH-properties would be most useful in situations where there are no simple composition series known. In these situations, before computing invariants, it might be good to know whether there are `simpler' composition series and a JH-property would tell us that we cannot do better.
\end{itemize}
\end{rem}

\subsection*{On Question \ref{Q:CharQH}}

We first remark that it is well-known that \emph{triangulated} composition series with composition factors $\cd^b(k)$ correspond to full exceptional sequences and vice versa, see Lemma \ref{L:ExcSeqCompSeries}. For derived categories of finite dimensional algebras these composition series will in general involve triangulated categories which are not derived module categories, see e.g. Example \ref{E:PropTriaCompSeries}. For derived categories of quasi-hereditary algebras there is always a (special) derived composition series with factors $\cd^b(k)$, see Corollary \ref{C:DerCompSeriesQH} -- namely, all involved recollements are given by primitive idempotents. Question \ref{Q:CharQH} asks about a converse. Liu \& Yang \cite{derivedsimple} showed that the derived categories of a finite dimensional algebra with at most two simple modules admits a derived composition series with factors $\cd^b(k)$ if and only if it is derived equivalent to a quasi-hereditary algebra. We show that this statement fails for algebras with more than two simples. More precisely, the factor algebra \begin{align}\label{E:A2}A_2:=A'_2/(\alpha_1 \alpha_2)\end{align} (see \eqref{E:CJHalgebra}) admits a derived composition series with factors $\cd^b(k)$ (Lemma \ref{L:Strat}), is not quasi-hereditary (see Lemma \ref{L:Notquasi}) and is the unique algebra in its derived equivalence class (up to Morita equivalence), see Corollary \ref{C:derivedisolated}. The key step is to show the following proposition which has further consequences and also relates to other work, see the remarks below.

\begin{prop}\label{P:Notequiv}
The algebra $A_2$ defined in \eqref{E:A2} is not derived equivalent to $A_{1}:=kQ_{1}/I_{1}$ where
\[
\begin{tikzpicture}[description/.style={fill=white,inner sep=2pt}]
    \matrix (n) [matrix of math nodes, row sep=1em, ampersand replacement=\&,
                 column sep=1.3em, text height=1.5ex, text depth=0.25ex,
                 inner sep=0pt, nodes={inner xsep=0.3333em, inner
ysep=0.3333em}]
    { Q_1:= \& 1 \& \& 2 \& \& 3   \\ 
    };
\path[-> ] (n-1-2) edge [bend left=20] node[ scale=0.75, fill=white] [midway] {$a$} 
(n-1-4);
\path[-> ] (n-1-2) edge [bend right=20] node[ scale=0.75, fill=white] [midway] {$b$} 
(n-1-4);
\path[-> ] (n-1-4) edge [bend left=20] node[ scale=0.75, fill=white] [midway] {$a$} 
(n-1-6);
\path[-> ] (n-1-4) edge [bend right=20] node[ scale=0.75, fill=white] [midway] {$b$} 
(n-1-6);

\draw (n-1-1) node[yshift=-29pt, xshift=17pt] {$I_{1}:=(a^2, b^2).$};
\end{tikzpicture}
\]


\end{prop}

\begin{rem} There seems to be no way of distinguishing the derived categories of $A_1$ and $A_2$ by known derived invariants.
Indeed, Euler forms are derived invariant by \cite[Proposition, p. 101]{Happelbook}. The matrices of the Euler forms of $A_1$ and $A_2$ in the bases given by the simple modules are 
\begin{align}
M_1=\begin{pmatrix} 1 & -2 & 2 \\ 0 & 1 & -2  \\ 0 & 0 & 1 \end{pmatrix}  \quad \text{     and  } \quad  M_2=\begin{pmatrix} 0 & -1 & 1 \\ 1 & 1 & -1 \\ -1 & 1 & 0 \end{pmatrix}. 
\end{align}
The corresponding integral bilinear forms are equivalent\footnote{This was observed by Ladkani  \cite{LadkaniNOTE}.} since $M_2=B^{\rm tr} M_1 B$ for
\begin{align}
B=\begin{pmatrix} -1 & 1 & 1 \\ -1 & 0 & 2 \\ 0 & 0 & 1 \end{pmatrix}. 
\end{align}

Moreover, Ladkani \cite{Ladkani} shows that the dimensions of the Hochschild cohomology groups of gentle algebras are completely determined by the corresponding derived invariants of Avella-Alaminos \& Gei{\ss} \cite{AG}. For both $A_1$ and $A_2$ this invariant can be computed to be $[2, 4]$.
 \end{rem}

\begin{rem}
The three algebras $A_1$, $A_2$ and $A'_2$, which play a key role in this article form a complete set of representatives of the derived equivalence classes for gentle algebras with three arrows and four vertices, see Corollary \ref{C:DerivedEq}.
\end{rem}

We point out some consequences and related work.

\smallskip

\noindent
$(a)$ Proposition \ref{P:Notequiv} is part of a conjecture by Bobi{\'n}ski \& Malicki \cite{BM} ($A_1=\Lambda'_0(1, 0)$ and $A_2=\Lambda_0(2, 1)$ in their notation, see also the paragraph after Lemma \ref{L:Kickout}). We also show that Proposition \ref{P:Notequiv} allows to distinguish the derived categories of a whole family of finite dimensional algebras confirming further parts of this conjecture, see Corollaries \ref{C:RedBobMali} \& \ref{C:Specialcase}. Recently, Amiot confirmed other cases for gentle algebras arising from a torus with one boundary component, see \cite{Amiot}. Upon receiving a preliminary version of this article Grzegorz Bobi{\'n}ski kindly informed us that using an extension of Amiot's techniques, he is able to establish the conjecture in all cases, see \cite{Bobinski}. Our alternative approach might nevertheless be useful to understand derived categories of gentle algebras, see also Remark \ref{R:Extend}.

 \noindent $(b)$ It follows from Avella-Alaminos \& Gei{\ss}' combinatorially defined derived invariants for gentle algebras \cite{AG} and considerations on the Euler form, that the derived equivalence class of $A_{1}$ contains at most $A_{2}$ (up to Morita-equivalence). 
In combination with Proposition \ref{P:Notequiv} this implies that $A_{1}$ and $A_{2}$ are `derived unique' algebras, i.e.~algebras for which the notions of Morita and derived equivalence coincide, see Definition \ref{D:derivedisolated} and Corollary \ref{C:derivedisolated} for details.

\noindent $(c)$ The algebra $A_1$ appears in several different places in the literature. We already mentioned \cite{BM}. Moreover, Burban \& Drozd show that $A_1$ is derived equivalent to coherent sheaves over a certain non-commutative irreducible nodal cubic curve \cite{BurbanDrozd} -- Burban also conjectured Corollary \ref{C:derivedisolated} for $A_1$. The algebra $A_1$ also appears in work of Seidel \cite[Section 3]{Seidel} in relation with a
Fukaya category of a certain Lefschetz pencil and in work of Kuznetsov \cite{Kuznetsov}\footnote{We thank Nathan Broomhead for giving us this reference.}
 on a geometric counterexample to the JH-property for triangulated composition series, see Remark \ref{R:Kuznetsov} (ii). Finally, Orlov \cite[Section 3.1.]{OrlovGeom} points out that $A_1$ corresponds to the Ising $3$-point function and is related to a Landau-Ginzburg model. Moreover, he shows that its derived category may be realized as a thick subcategory of a strong exceptional collection of vector bundles on a three dimensional smooth projective variety\footnote{We thank Theo Raedschelders for pointing out this reference.}.

\subsection*{Structure}
Section \ref{S:RecollementsETC} contains well-known background material: on recollements (and their relation to admissible subcategories), which can be arranged into triangulated composition series leading to the classical notion of exceptional sequences (in particular, we consider the case of derived categories of quasi-hereditary algebras with the exceptional sequences of standard and costandard modules). We also provide examples showing that complete exceptional sequences need not be full and that full exceptional sequences don't give \emph{derived} composition series in general.

We explain our constructions of quasi-hereditary algebras with derived composition series of different length in Section \ref{S:JH}. This gives a negative answer to Question \ref{Q:JordanHoelder}. The results of this section are not needed in the rest of the text.

In Section \ref{S:BobMali}, we use Proposition \ref{P:Notequiv}, which we prove in the final Section \ref{S:Proof}, to describe the derived equivalence classes of gentle algebras with three vertices and four arrows. This implies that the algebra $A_2$ is not derived equivalent to a quasi-hereditary algebra, which leads to a negative answer to Liu \& Yang's Question \ref{Q:CharQH} in Section \ref{S:LY}. In Section \ref{S:BobMali}, we also show how to reduce the conjecture of Bobi{\'n}ski \& Malicki to `algebras with full relations' and apply this reduction to obtain further parts of this conjecture from Proposition \ref{P:Notequiv}. We include some background material on the work of Avella-Alaminos \& Gei{\ss} \cite{AG}, which is a key ingredient both in our reduction argument and in the  proof of Proposition \ref{P:Notequiv}, see subsection `Derived equivalences and the AG-invariant'. 

\subsection*{Notation}

Throughout, let $k$ be an algebraically closed field. All modules are left modules. For a $k$-algebra $A$, we denote the derived category of left $A$-modules $\cd(A-\Mod)$ by $\cd(A)$ and the bounded derived category of finitely generated left $A$-modules $\cd^b(A-\mod)$ by $\cd^b(A)$.
For a set of objects $\cs$ in a triangulated category $\ct$, we write the thick subcategory generated by $\cs$ as $\thick(\cs)$. 
We read elements in the path algebra $kQ$ of a quiver $Q$ from right to left. 

\section{Preliminaries: recollements, composition series and exceptional sequences} \label{S:RecollementsETC}

The following notion is classical, see e.g. \cite{Bondal90}.

\begin{defn} \label{D:adm}
Let $\ct$ be a triangulated category. A full triangulated subcategory $\ca \subseteq \ct$ is called \emph{admissible} if the natural inclusion admits both a left and a right adjoint functor. In particular, $\ca$ is \emph{thick}, i.e. it is closed under taking direct summands, see e.g. \cite[Proposition 1.6]{BondalKap}.
\end{defn}

For a subset $\cs$ of a triangulated category $\ct$, we define the \emph{right orthogonal subcategory} $\cs^\perp:=\{ÊT \in \ct \mid \Hom_\ct(S, T[i])=0 \text{ for all }ÊS \in \cs \, \text{Êand all }Êi \in \ZZ\}$, which is a triangulated subcategory of $\ct$. The \emph{left orthogonal subcategory} $^\perp\!\cs$ is defined dually.

\begin{rem}
For an admissible subcategory $\ca$ the right orthogonal $\ca^\perp$ is left admissible (i.e. the natural inclusion has a left adjoint) but need not be right admissible in general, see e.g. \cite[Example 3.4.]{AKLY}. Dually, $^\perp\!\!\ca$ is right admissible but not left admissible in general. This is closely related to the notion of semi-orthogonal decompositions, see for example \cite[p.3, conventions on recollements]{HKPrep}.
\end{rem}

For an admissible subcategory $\ca$, the corresponding quotient $\ct \to \ct/\ca$ has good properties (see e.g. \cite[Section 1.4.4]{BBD} and also  \cite[Section 9]{Neeman99}) leading to the notion of a recollement.

\begin{prop}\label{P:Recollfromadm}
Let $\ca \subseteq \ct$ be admissible. Then the following statements hold.
\begin{itemize}
\item[(a)] The canonical triangulated quotient functor $j^*\colon \ct \to \ct/\ca$ admits both a left adjoint $j_!$ and a right adjoint $j_*$. This gives rise to a \emph{recollement}, i.e a sequence
\begin{align} \label{E:RecollDef}
\begin{xy}
\SelectTips{cm}{}
\xymatrix{\ca \ar[rr]|{i_*=i_!}&&\ct\ar[rr]|{j^!=j^*}\ar@/^15pt/[ll]^{i^!}\ar@/_15pt/[ll]_{i^*}&&\ct/\ca \ar@/^15pt/[ll]^{j_*}\ar@/_15pt/[ll]_{j_!}}
\end{xy}
\end{align}
where $i_*$ is the canonical inclusion with left adjoint $i^*$ and right adjoint $i^!$. 

\item[(b)] Conversely, let $j^*\colon \ct \to \cq$ be a triangulated quotient functor (i.e. $j^*$ induces and equivalence $\ct/\ker j^* \to \cq$) with left adjoint $j_!$ and right adjoint $j_*$. Then $\ker j^* \subseteq \ct$ is admissible.

\item[(c)] The right adjoint $j_*$ induces a triangle equivalence $j_* \colon \ct/\ca \rightarrow \ca^\perp$. Dually, the left ajoint $j_!$ yields a triangle equivalence $j_! \colon \ct/\ca \rightarrow ^\perp \! \!\!\ca$. In particular, $j_*$ and $j_!$ are fully faithful.
\end{itemize}
\end{prop}
\begin{proof}
Part (a) \& (b)  are  \cite[Proposition 9.1.18]{Neeman99} and its dual. To see part (c), we note that the composition $\gamma$ of the  inclusion 
$ \ca^\perp \subseteq \ct$ followed by the natural projection $j^*\colon \ct \to \ct/\ca$ is an equivalence, see e.g.  \cite[Proposition 9.1.16]{Neeman99}. Using the adjunction $(j^*, j_*)$ one can check that $j_*\colon \ct/\ca \to \ca^\perp$ is well-defined and right adjoint to the equivalence $\gamma\colon \ca^\perp \to \ct/\ca$. Since right adjoints are unique $j_*$ has to be a quasi-inverse to $\gamma$ completing the proof.
\end{proof}

\begin{rem}
Parts (a) \& (b) show that  recollements are completely determined by fixing `one half'. 
\end{rem}

\begin{rem}\label{R:Standardseq}
Consider a recollement as in \eqref{E:RecollDef}.
It is well-known (see e.g. \cite[p. 319]{Neeman99}) that every object $X$ of $\ct$ fits into two distinguished triangles
\[i_!i^!X \rightarrow X \rightarrow j_*j^*X \rightarrow i_!i^!X[1] \, \,
\text{ and } \, \,  j_!j^!X \rightarrow X \rightarrow
i_*i^*X \rightarrow j_!j^!X[1],\] where the morphisms starting from and
ending at $X$ are the units and counits of the adjunctions.
\end{rem}

\begin{defn}
An object $E$ in a $k$-linear triangulated category is called \emph{exceptional}, if 
\[\bigoplus_{i \in \ZZ} \Hom_\ct(E, E[i])= k. \]
\end{defn}

\noindent
By Proposition \ref{P:Recollfromadm}, the following well-known example yields a recollement starting from the left hand side. This is used in Lemma \ref{L:ExcSeqCompSeries} to construct recollements from exceptional sequences.

\begin{ex} \label{E:admissible}
Let $\ct=\cd^b(A)$ for a finite dimensional $k$-algebra $A$ of finite global dimension. Let $E$ in $\ct$ be an exceptional object. Then the thick subcategory $\thick(E) \subseteq \ct$ generated by $E$ is admissible. Indeed the right adjoint is given by $\RHom_A(E, -) \overset{\mathsf L}{\otimes}_k E$ and the left adjoint is $\RHom_A(-, E)^* \overset{\mathsf L}{\otimes}_k E$, where $(-)^*=\Hom_k(-, k)$ denotes the $k$-duality. More generally, one can replace $\ct$ by a $k$-linear algebraic triangulated category, such that $\dim_k\bigoplus_{i \in \ZZ}Ê\Hom_\ct(X, Y[i]) < \infty$ for all $X, Y$ in $\ct$.
\end{ex}

In combination with  Proposition \ref{P:Recollfromadm}, the following well-known proposition gives examples for recollements starting from a fixed right hand side. 

\begin{prop}\label{P:AdjTriple}
Let $A$ be a finite dimensional $k$-algebra and $e \in A$ be an idempotent such that $eAe$ has finite global dimension. Then there is a triple of adjoint triangle functors
\begin{align}\label{E:AdjTriple}
\begin{xy}
\SelectTips{cm}{}
\xymatrix{\cd^b(A)\ar[rrrr]|{\begin{smallmatrix} j^!= \RHom_{A}(Ae, -) \\ = \\ j^*=eA\lten_{A} - \end{smallmatrix} }&&&&\cd^b(eAe)\ar@/_25pt/[llll]_{j_!=Ae \lten_{eAe} -}\ar@/^25pt/[llll]^{ j_* =\RHom_{eAe}(eA, - )}}
\end{xy},
\end{align}
i.e.~$(j_{!}, j^!)$ and $(j^*, j_{*})$ are adjoint pairs. Moreover,  $j^*=j^!$ is a triangulated quotient functor.
\end{prop}
\begin{proof}
Already on the abelian level, we have a triple of adjoint functors
\begin{align*}
\begin{xy}
\SelectTips{cm}{}
\xymatrix{A-\mod \ar[rrrr]|{\begin{smallmatrix} \Hom_{A}(Ae, -) \\ = \\ eA\ten_{A}- \end{smallmatrix} }&&&& eAe-\mod \ar@/_25pt/[llll]_{Ae\ten_{eAe} -}\ar@/^25pt/[llll]^{ \Hom_{eAe}(eA, - )}}
\end{xy},
\end{align*}
by the adjunction formula. Deriving this and using $\gldim eAe < \infty$ yields the adjoint triple \eqref{E:AdjTriple} above. To show that $j^!$ is a quotient functor one can proceed as follows. Using the dual of \cite[Proposition III.5]{Gabriel62} in combination with \cite[Theorem 8.4.4.]{DrozdKirichenko} shows that $\Hom_{A}(Ae, -)$ induces is an equivalence of abelian categories $A-\mod/(A/AeA)-\mod \to eAe-\mod$, see e.g. \cite[Proposition 5.9]{Miyachi}. In combination with \cite[Theorem 3.2]{Miyachi} this finishes the proof.\end{proof}
\begin{rem} \label{R:Recollidemp}
One can check that the kernel of $j^!$ is $\thick(A/AeA-\mod)$. So combining Proposition \ref{P:AdjTriple} with Proposition \ref{P:Recollfromadm} yields a recollement $(\thick(A/AeA-\mod), \cd^b(A), \cd^b(eAe))$.
\end{rem}


Viewing recollements as analogues of short exact sequences for triangulated categories, leads to the notions of triangulated simple categories and triangulated composition series, which we introduce below. The main results of this article deal with the special case of derived simple categories and derived composition series, see Definition \ref{D:Strati}. However, for examples from algebraic geometry and some general statements (e.g. Lemma \ref{L:ExcSeqCompSeries})  it is convenient to introduce this terminology.

\begin{defn}
A triangulated category $\ct$ is called \emph{triangulated simple} if there is no non-trivial recollement $(\ct', \ct, \ct'')$.
\end{defn} 

\begin{ex}\label{E:derivedsimple}
\begin{itemize}
\item[(a)] The bounded derived category of vector spaces $\cd^b(k)$ is triangulated simple. Indeed more generally triangulated categories which do not admit non-trivial thick subcategories are triangulated simple.
\item[(b)] Indecomposable Calabi-Yau categories $\cc$ (e.g. derived categories of connected Calabi-Yau varieties, cluster categories, singularity categories of isolated Gorenstein singularities) are triangulated simple. Indeed assume that there exists a non-trivial recollement, i.e. an admissible subcategory $\ca \subset \cn$. It follows from the triangles in Remark \ref{R:Standardseq} that $\cc$ is generated by $\ca$ and $\ca^{\perp}$. Using the Calabi-Yau property, we see that $\ca^\perp=^\perp\!\!\ca$ and therefore $\cc \cong \ca \oplus\ca^{\perp}$.
\end{itemize}
\end{ex}

\begin{defn}\label{D:triangcompseries}
Let $\ct$ be a triangulated category. A \emph{triangulated composition series} of $\ct$ is a binary tree constructed by iteratively taking  recollements. Starting with a recollement $(\ct_0, \ct, \ct_1)$ of $\ct$ and continuing with recollements of $\ct_0$ and $\ct_1$ and so forth until triangulated simple categories are reached.
\end{defn}

We refer to Example \ref{E:PropTriaCompSeries}, Lemma \ref{L:ExcSeqCompSeries} and Section \ref{S:JH} for examples.

\subsection*{Exceptional sequences and derived categories of quasi-hereditary algebras}

\begin{defn}
Let $\ct$ be a $k$-linear triangulated category. A sequence $(E_1, \ldots, E_n)$ of exceptional objects $E_i$ is called \emph{exceptional sequence} if 
\begin{align}
\Hom_\ct(E_j, E_i[s])=0 \text{  for all } \quad  j > i \quad \text{Êand all } \quad s \in \ZZ.
\end{align}
It is called \emph{full}Ê  if $\thick(E_1, \ldots, E_n)= \ct$ and \emph{complete}Ê  (or \emph{maximal}) if there exists no exceptional object $E$ in $\ct$ such that $(E_1, \ldots, E_{i-1}, E, E_{i}, \ldots, E_n)$ is an exceptional sequence, where $1 \leq i \leq n+1$.
\end{defn}

\begin{rem}\label{R:Kuznetsov}
\begin{itemize}
\item[(i)] It is well-known that full exceptional sequences are complete. We proceed by induction. Let $\ct=\thick(E)$ for an exceptional object $E$. Then $\ct \cong \cd^b(k)$ where $E$ is identified with $k$ - in particular, any object in $\ct$ is a direct sum of shifts of $E$. Therefore, this exceptional sequence cannot be extended. Assume that the statement is already shown for a full exceptional sequence of length at most $n-1$. Let $(E_1, \ldots, E_n)$ be a full exceptional sequence and assume that there is an exceptional sequence $(E_1, \ldots, E_{i-1}, E, E_{i}, \ldots, E_n)$, where $1 \leq i \leq n+1$. This yields the following equalities of subcategories
\begin{align*}
\thick(E_i, \ldots, E_n) = ^\perp\! \thick(E_1, \ldots, E_{i-1})=\thick(E, E_i, \ldots, E_n)
\end{align*}
see e.g. \cite[Lemma 6.1]{Bondal90}. By definition, $(E_i, \ldots, E_n)$ is a full exceptional sequence in this subcategory. By induction it is complete which contradicts the existence of the full exceptional sequence $(E, E_i, \ldots, E_n)$. This finishes the proof.
\item[(ii)] The converse fails already for the derived category of the algebra $A_1$ from Proposition \ref{P:Notequiv}, which has global dimension $2$ and hence is quasi-hereditary \cite[Theorem 2]{DR} (we will see later (Corollary \ref{C:DerCompSeriesQH}) that this implies that $\cd^b(A)$ admits a full exceptional sequence). Bondal (see e.g. \cite{Kuznetsov}) observed that the exceptional collection $(E)$ of length $1$ is complete, where $E=1 \xrightarrow{a} 2 \xrightarrow{b} 3$ is an exceptional $A_1$-module. To see this one can check that the Euler forms on $^\perp\!E$ and $E^\perp$ are anti-symmetric and therefore these categories don't contain exceptional objects. This is sometimes referred to as a `failure of Jordan-H{\"o}lder Theorem' for semi-orthogonal decompositions (or triangulated composition series in our language) and was used by Kuznetsov \cite{Kuznetsov} to construct new geometric counter-examples to the Jordan-H{\"o}lder property. 

For piecewise hereditary algebras (i.e. finite dimensional algebras which are derived equivalent to abelian categories of global dimension $1$) the notions of full and complete exceptional sequences coincide, see e.g. \cite{AKLJH} together with Lemma \ref{L:ExcSeqCompSeries} and also \cite{Crawley-Boevey} Êfor the special case of exceptional sequences of quiver representations. 
\end{itemize}
\end{rem}

It follows from Example \ref{E:admissible} \& Proposition \ref{P:Recollfromadm} that full exceptional sequences give rise to triangulated composition series with composition factors $\cd^b(k)$ and vice versa. This is summarized in the following well-known lemma.

\begin{lem}\label{L:ExcSeqCompSeries}
Let $\ct$ be a $k$-linear algebraic triangulated category, such that $\dim_k \bigoplus_{i \in \ZZ}Ê\Hom_\ct(X, Y[i]) < \infty$ for all $X, Y$ in $\ct$. Then every full exceptional sequences in $\ct$ gives rise to a triangulated composition series with factors $\cd^b(k)$. Conversely, every such composition series yields a full exceptional sequence.
\end{lem}
\begin{proof}
Let $(E_1, \ldots, E_n)$ be a full exceptional sequence. Example \ref{E:admissible} shows that $\thick(E_1)$ is an admissible subcategory. Using Proposition \ref{P:Recollfromadm} (a), we get a recollement $(\thick(E_1), \ct, \ct/\thick(E_1))$, where $\ct/\thick(E_1) \cong ^\perp\!\! E_1$ by Proposition \ref{P:Recollfromadm} (c) and further $^\perp\! E_1 = \thick(E_2, \ldots, E_n)$ by \cite[Lemma 6.1]{Bondal90}. By induction, we get a composition series with factors $\thick(E_i) \cong \cd^b(k)$ (see Example \ref{E:derivedsimple} (a)).

Conversely, assume that $\ct$ has a composition series with factors $\cd^b(k)$. In particular, we obtain a recollement of the form
$(\cd^b(k), \ct, \ct')$ or $(\ct'', \ct, \cd^b(k))$. In both cases the image $E$ of $k \in \cd^b(k)$ in $\ct$ is exceptional. Proposition  \ref{P:Recollfromadm} (c) shows that $\ct'$ respectively $\ct''$ identify with $E^\perp$ and $^\perp\!E$. By assumption these categories again admit recollements involving $\cd^b(k)$ as one of the factors. Iterating this process yields an exceptional sequence. The standard triangles associated with recollements (Remark \ref{R:Standardseq}) imply that this sequence is full. 
\end{proof}

We turn to an example which will be important in the sequel. 
Derived categories of quasi-hereditary algebras admit full exceptional sequences, for example given by standard modules. We start
with the definition of a quasi-hereditary algebra due to Scott \cite{Scott} (cf. \cite{derivedsimple}).

\begin{defn} \label{D:QH}
Let $A$ be a finite dimensional $k$-algebra and let $e \in A$ be an idempotent. The two-sided ideal $AeA$ is called \emph{heredity} if $eAe$ is a semi-simple algebra and $AeA$ is projective as a left $A$-module. The algebra $A$ is called \emph{quasi-hereditary} Êif there exists a chain of two-sided ideals
\begin{align} \label{E:chain}
0=J_0 \subseteq J_1 \subseteq J_2 \subseteq \ldots \subseteq J_m=A
\end{align}
 such that $J_i/J_{i-1}$ is a heredity ideal in $A/J_{i-1}$ for all $i=1, \ldots, m$. In particular, semi-simple algebras and all quotient algebras $A/J_i$ are quasi-hereditary.
\end{defn}

\begin{rem} \label{R:QH}
\begin{itemize}
\item[(a)] Let $A$ be a quasi-hereditary algebra. One can refine the chain \eqref{E:chain} in such a way that all heredity ideals $J_i/J_{i-1}$ are given by primitive idempotents. 
\item[(b)] Let $A$ be a finite dimensional $k$-algebra and let $e \in A$ be an idempotent. The canonical functor $\iota\colon \cd^b(A/AeA) \to \cd^b(A)$ is fully faithful, if $AeA$ is a projective left $A$-module, see \cite[Theorem 3.1(1)]{CPS}. In this situation, $\im \iota = \thick(A/AeA-\mod)$ is the kernel of the quotient functor $\Hom_A(Ae, - ) \colon \cd^b(A) \to \cd^b(eAe)$. If $eAe$ has finite global dimension, we obtain a recollement $(\cd^b(A/AeA), \cd^b(A), \cd^b(eAe))$ by Proposition \ref{P:AdjTriple}.  
\end{itemize}
\end{rem}

Combining Remarks \ref{R:QH} (a) and (b) with Definition \ref{D:QH} shows the following well-known lemma, see e.g. \cite[Corollary 3.7]{CPS2}.

\begin{lem} \label{L:RecollQH}
Let $A$ be a quasi-hereditary algebra. Then the primitive idempotents $e_1, \ldots, e_n \in A$ may be ordered such that
there are recollements $(\cd^b(A/A(e_1+\ldots+e_{i+1})A), \cd^b(A/A(e_1+\ldots+e_i)A), \cd^b(k))$. 
\end{lem}

\begin{cor} \label{C:DerCompSeriesQH}
Let $\ct=\cd^b(A)$ for a quasi-hereditary algebra $A$. Then there exists a triangulated composition series with factors $\cd^b(k)$. In particular, this yields many full exceptional sequences by Lemma \ref{L:ExcSeqCompSeries}. Namely in each recollement we can either choose the embedding $j_!=Ae \lten_{eAe} -$ or $j_* =\RHom_{eAe}(eA, - )$. One can check that the sequence of standard modules arises from always choosing $j_!$ and the sequences costandard modules arises from always chosing $j_*$.
\end{cor}

In the situation of Lemma \ref{L:RecollQH}, all triangulated categories appearing in the triangulated composition series are derived modules categories. In a series of papers Angeleri H{\"u}gel, K{\"o}nig, Liu (and Yang) \cite{AKL, AKLJH, AKLY} studied triangulated composition series of this form. We introduce some terminology for later use.

\begin{defn}\label{D:Dsimple}
A finite dimensional $k$-algebra $A$ is called \emph{derived simple} if there exists no non-trivial recollement $(\cd^b(A_1), \cd^b(A), \cd^b(A_2))$ with finite dimensional $k$-algebras $A_1$ and $A_2$.
\end{defn}

\begin{rem}
Every triangulated simple algebra is derived simple.
\end{rem}

Following \cite[Section 5]{AKL} we can now introduce the notion of composition series of derived module categories -- these can be thought of as analogues of composition series for modules over rings.

\begin{defn}\label{D:Strati}
A \emph{composition series of the derived module category} $\cd^b(A)$ of a finite dimensional $k$-algebra $A$ is a triangulated composition series (see Definition \ref{D:triangcompseries}) such that all triangulated categories appearing in the binary tree are equivalent to derived categories of finite dimensional algebras. It is also called \emph{derived composition series} of $A$.
\end{defn}


\noindent

The following example shows that full exceptional sequences need not give rise to \emph{derived} composition series. In fact the exceptional sequence studied here also leads to our counterexample for the question of Liu \& Yang, see Section \ref{S:LY}. 

\begin{ex}\label{E:PropTriaCompSeries}
Let $A=A_2$ be the algebra from Proposition \ref{P:Notequiv} and consider the full exceptional sequence of $A$-modules $(S_2, P_1, P_3)$, which by Proposition \ref{P:Recollfromadm} gives rise to a recollement $(\thick(P_3), \cd^b(A), P_3^\perp)$, and $P_3^\perp \cong \thick(S_2 \oplus P_1)$. One can check that there is an isomorphism of graded algebras
\begin{align}
\bigoplus_{s \in \ZZ}Ê\Hom_{\cd^b(A)}(S_2 \oplus P_1, (S_2 \oplus P_1)[s]) \cong k\left( 
\begin{xy}
\SelectTips{cm}{10}
\xymatrix{x \ar@/^/[r]|{\, 0 \, } \ar@/_/[r]|{\, 2 \,}  & y.   } 
\end{xy} 
\right) =: G
\end{align}
\end{ex}

where one arrow is in degree $0$ and the other arrow is in degree $2$. By definition, the graded algebra $G$ is isomorphic to the 
cohomology of the dg endomorphism algebra $\mathcal{E}nd(S_2 \oplus P_1)$, which can be equipped with a minimal $A_\infty$-structure such that there is an $A_\infty$-quasi-isomorphism (see \cite{Kadeishvili80} and also \cite[Section 3.3.]{KellerA} and references in there) 
\begin{align} \label{E:gradedExt}
\mathcal{E}nd(S_2 \oplus P_1) \cong H^*(\mathcal{E}nd(S_2 \oplus P_1)).
\end{align}
Since the quiver of $H^*(\mathcal{E}nd(S_2 \oplus P_1)) \cong G$ is directed and has only two vertices this $A_\infty$-algebra has no higher multiplications, see for example \cite[Section 3.5]{KellerA}. This shows that there is a quasi-isomorphism of dg algebras $\mathcal{E}nd(S_2 \oplus P_1) \cong G$, where $G$ is considered as a dg algebra with trivial differential. In combination with Keller's Morita theorem for triangulated categories (see e.g. \cite[Theorem 3.8 b)]{Keller06d}) this yields triangle equivalences 
\begin{align}
\thick(S_2 \oplus P_1) \cong \per \mathcal{E}nd(S_2 \oplus P_1) \cong \per G.
\end{align}
One can show that $\per G$ is not triangle equivalent to the derived category of a $k$-algebra, see e.g. \cite[Corollary 3.3]{derivedsimple}.

\section{Quasi-hereditary algebras with non-unique derived composition series}\label{S:JH}
 The aim of this section is to construct a family of examples having different derived composition series. More precisely, given any natural number $n$ we construct a finite dimensional quasi-hereditary algebra $A$ such that $\cd^*(A)$ has at least $2^n$ derived composition series of pairwise different length and at least $n$ different derived simple factors occur. In particular, this gives a negative answer to the uniqueness part of Question \ref{Q:JordanHoelder} for quasi-hereditary algebras. The results of this section are not needed in the rest of the text. In this section, we write  $\cd^*(A)$ for the derived categories
$\cd(A-\Mod)$, $\cd^-(A-\Mod)$, $\cd^b(A-\Mod)$, $\cd^b(A-\mod)$, $K^b(\proj-A)$ of left $A$-modules.

\subsection*{Examples arising from generalized Fibonacci algebras}
\noindent
For $l \in \ZZ_{\geq 1}$, we consider a family of algebras $B_l:=kQ_l/I_l$ given by quivers $Q_l$ 
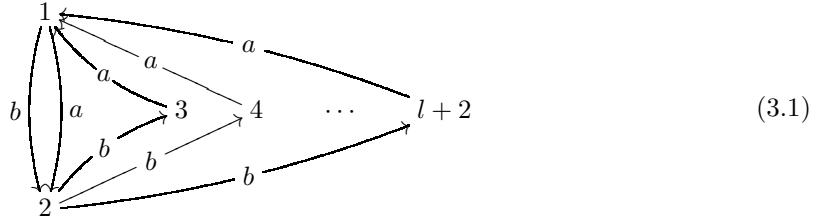
\begin{align}\label{Counterex}
\begin{array}{c}
\begin{xy}
\SelectTips{cm}{10}
\xymatrix@C=17pt{
1 \ar@/^/[dd]^{\displaystyle a }  \ar@/_/[dd]_{ \displaystyle b  }\\
&& 3  \ar@/^/[llu]|{\displaystyle \,\, a \,\, } & 4 \ar[lllu]|{\displaystyle \,\, a \,\, } & \cdots & l+2 \ar@/_/[lllllu]|{\displaystyle \,\, a \,\, }\\
 2 \ar@/^/[rru]|{\displaystyle \,\, b \,\, } \ar[rrru]|{\displaystyle \,\, b \,\, } \ar@/_/[rrrrru]|{\displaystyle \,\, b \,\, } 
}
\end{xy} 
\end{array}
\end{align}
with relations  $I_{l}:=(ba)$.

Note that $B_1=A'_2$ is the algebra from the introduction \eqref{E:CJHalgebra}.
We show below that the algebras $B_l$ are quasi-hereditary of global dimension $2$ and that they give a negative answer to Question \ref{Q:JordanHoelder}. In other words, they do not satisfy the derived Jordan-H\"older property as studied by Angeleri-H\"ugel, K\"onig \& Liu \cite{AKL}.

\begin{lem}\label{L1}
Let $l \in \ZZ_{\geq 1}$, set $B=B_l$ and let $e=e_1 + e_2 \in A$.
\begin{itemize}
\item[(a)] $B$ is a finite dimensional $k$-algebra.
\item[(b)] $\gldim B =2$.
\item[(c)] $\Ext^i_{B}(B/BeB, B/BeB)=\Ext^i_{B}(S_3 \oplus \ldots \oplus S_{l+2}, S_3 \oplus \ldots \oplus S_{l+2})=0$ for $i>0$.
\end{itemize}
\end{lem}
\begin{proof}
\begin{itemize}
\item[(a)] One can check that all paths of length greater than $4$ are contained in $I_l$.
\item[(b, c)] The projective resolutions of the simple $B_l$ modules $S_i$ are given as follows:
\begin{align*}
0 \to P_{l+2} \oplus \ldots \oplus P_3 \rightarrow P_2 \oplus P_2 \to P_1 \to S_1& \to 0 \\
0 \to P_{l+2} \oplus \ldots \oplus P_3 \to P_2 \to S_2& \to 0 \\
0 \to P_2 \rightarrow  P_1 \to P_{i+2} \to S_{i+2}& \to 0 \qquad \text{for }Êi=1, \ldots, l.  \\
\end{align*}
This yields (b) and (c).\end{itemize}\end{proof}

\begin{cor}\label{QH}
Let $l \in \ZZ_{\geq 1}$ then $B_l$ is a quasi-hereditary algebra.
\end{cor}
\begin{proof}
This follows from Dlab \& Ringel \cite[Theorem 2]{DR} in combination with Lemma \ref{L1} (b).
\end{proof}

\begin{lem}\label{corner}
Let $l \in \ZZ_{\geq 1}$ and let $e=e_1 + e_2 \in B$. Then $eB_le$ is given by the quiver
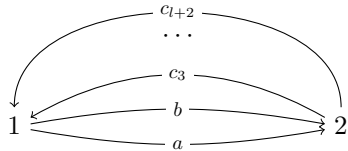
\[
\begin{tikzpicture}[description/.style={fill=white,inner sep=2pt}]
    \matrix (n) [matrix of math nodes, row sep=1em,
                 column sep=2.25em, text height=1.5ex, text depth=0.25ex,
                 inner sep=0pt, nodes={inner xsep=0.3333em, inner
ysep=0.3333em}]
    {   && \cdots \\ \\ 
        1 &&&& 2 \\
    };
\draw[->] (n-3-1) edge [bend left=10] node[scale=0.75, fill=white] [midway] {$b$} (n-3-5);
\draw[->] (n-3-1) edge [bend right=10] node[ scale=0.75, fill=white] [midway] {$a$}  (n-3-5);

\draw[<-] (n-3-1) edge [bend left=30] node[scale=0.75, fill=white] [midway] {$c_3$} (n-3-5);
\draw[<-] (n-3-1) edge [bend left=90] node[ scale=0.75, fill=white] [midway] {$c_{l+2}$}  (n-3-5);
 
\end{tikzpicture}
\]
with relations $bc_i$ and $c_ja$. Moreover, these algebras are isomorphic to the generalized Fibonacci algebra $G_l:=A_3((1, 1), (l))$ as studied by Liu \& Yang \cite{derivedsimple} and it is shown in {\it loc. cit.} that the $G_l$ have global dimension $3$.
\end{lem}

\begin{prop}\label{CEx}
Let $l \in \ZZ_{\geq 1}$ then $B=B_l$ has at least two non-equivalent derived composition series. 
\begin{itemize}
\item[(a)] A derived composition series of length $l+2$ with all composition factors  given by $\cd^*(k)$.
\item[(b)] A derived composition series of length $l+1$ with $l$ composition factors  given by $\cd^*(k)$ and one composition factor given by $\cd^*(G_l)$.
\end{itemize}
In particular, the derived JH property fails for these algebras.
\end{prop}
\begin{proof}
The existence of (a) follows from the fact that $A_l$ is quasi-hereditary (Corollary \ref{QH}) and that quasi-hereditary algebras admit derived composition series with all factors of the form $\cd^b(k)$, see Corollary \ref{C:DerCompSeriesQH}.

In order to see the existence of (b), let $e=e_1+e_2$. There is an algebra isomorphism $B/BeB \cong \prod_{i=1}^l k$. By Lemma \ref{corner}, $eBe$ has finite global dimension and therefore Proposition \ref{P:AdjTriple} yields a recollement $(\thick(B/BeB-\mod), \cd^b(B), \cd^b(eBe))$, see e.g. Remark \ref{R:Recollidemp}. Lemma \ref{L1} (c) shows that the natural functor 
$\cd^b(B/BeB) \to \cd^b(B)$ induces an equivalence $\cd^b(B/BeB) \cong \thick(B/BeB-\mod)$. So we get a recollement $(\cd^b(B/BeB), \cd^b(B), \cd^b(eBe))$. Since all involved algebras have finite global dimension (see Lemma \ref{L1} (b)), we get recollements $(\cd^*(B/BeB), \cd^*(B), \cd^*(eBe))$ by \cite{AKLY}.
Combining the algebra isomorphisms $B/BeB \cong  \prod_{i=1}^l k$ and $eBe \cong G_l$ (see Lemma \ref{corner}) with the fact that $G_l$ is derived simple (see \cite[Corollary 4.3]{derivedsimple}) completes the proof.
\end{proof}

\begin{rem}
 Dong Yang pointed out that Vossieck's derived discrete algebras (of finite global dimension) satisfy the derived JH property. Recently, this was also shown in a work of Yongyun Qin for all derived discrete algebras, see \cite{Qin}. The algebra $B_1=A'_2$ is gentle and therefore its derived category has tame representation type. In this sense our example is as small as possible.
\end{rem}

The idea which we used to modify the derived simple algebras $G_k$ to obtain quasi-hereditary algebras $A_k$ seems to work for all derived simple two-vertex algebras of finite global dimension, see Liu \& Yang \cite{derivedsimple} for a list of these algebras. This leads to the following question.

\begin{Question}
Let  $G$ be a finite dimensional derived simple algebra of finite global dimension. Does there exists a quasi-hereditary algebra $A$ such that $\cd^*(G)$ occurs as a factor in a derived composition series of $\cd^*(A)$ ? 
\end{Question}

In the next paragraph, we `glue' copies of the algebras $B_l$ together using certain triangular matrix algebras. The algebras constructed in this way are again quasi-hereditary. Extending and building on Proposition \ref{CEx}, we show that they can have an arbitrary number of derived composition series of different length (Proposition \ref{P:MainJH}).

\subsection*{Glueing}\label{S:Glue}


Let $B'=kQ'/I'$ and $B''=kQ''/I''$ be finite dimensional algebras and let $a \in Q'_0$ and $b \in Q''_0$ be vertices. We write $B=B' \, _a\!\!\to_b B''$ for the triangular matrix algebra 
\begin{align}\label{E:glue}
\begin{pmatrix}
B'' & B''e_b \otimes_k e_aB' \\
0  & B' 
\end{pmatrix},
\end{align}
which may also be written as $B=kQ/I$, where $Q$ is obtained from the disjoint union of $Q'$ and $Q''$ by adding an arrow $a \to b$ and $I := I' + I''$. This construction is associative: given quiver algebras $A=kQ/I$, $B=kR/J$, $C=kS/K$ and vertices $a \in Q_0$, $b, b' \in R_0$ and $c \in S_0$, we have \begin{align*}
\left(A  \, _a\!\!\to_b B \right)\, _{b'} \!\!\to_c C \cong  A  \, _a\!\!\to_b \left(B \, _{b'} \!\!\to_c C \right),
\end{align*}
which we write as $A \, _a\!\!\to_b B \, _{b'}\!\!\to_c C$ in the sequel.
 
 Consider the idempotent $e_{B'}=\sum_{i \in Q'_0} e_i \in B $ in \eqref{E:glue}. Then we have the following fact.

\begin{lem}\label{L37}
There are algebra homomorphisms $B' \cong e_{B'} B e_{B'}$ and $B'' \cong B/Be_{B'}B$. Moreover $B/Be_{B'}B$ is a projective $B$-module. In particular, $\Ext^i_B(B/Be_{B'}B, B/Be_{B'}B)=0$ for all $i>0$. 
\end{lem}  

\begin{cor}\label{C:Recoll}
Assume that $B'$ and $B''$ have finite global dimension. Then there is a recollement $(\cd^*(B''), \cd^*(B), \cd^*(B'))$.  
\end{cor}
\begin{proof}
Using Lemma \ref{L37} the existence of a recollement $(\cd^b(B''), \cd^b(B), \cd^b(B'))$
follows as in the proof of Proposition \ref{CEx} (b). In particular, $\gldim B', \gldim B'' < \infty$ imply $\gldim B < \infty$, cf. e.g. \cite[Theorem I]{AKLY}. Therefore we get all recollements $(\cd^*(B''), \cd^*(B), \cd^*(B'))$, e.g. by \cite{AKLY}.
\end{proof}

We introduce the following notation:

\begin{defn}\label{D:39}
Let $\mathbf{l}=(l_1, \ldots, l_m)$ be a sequence of positive integers. 
Set \[B_\mathbf{l}:= B_{l_1} \, _2\!\!\to_1 B_{l_2} \, _2\!\!\to_1 \ldots  \, _2\!\!\to_1 B_{l_m},\] where the $B_{l_i}$ are the algebras defined in \eqref{Counterex}.
\end{defn}

\begin{ex} \label{E:CEJHA}
The algebra $B_{3, 2, 4}$ is given by the following quiver with relations $I_{3, 2, 4}=(ba)$
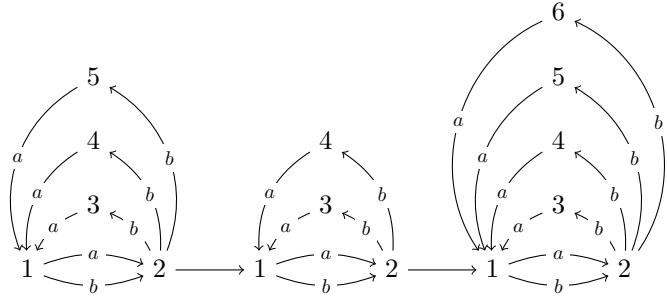
\begin{align*}
\begin{tikzpicture}[description/.style={fill=white,inner sep=2pt}]
    \matrix (n) [matrix of math nodes, row sep=1em, ampersand replacement=\&,
                 column sep=1.3em, text height=1.5ex, text depth=0.25ex,
                 inner sep=0pt, nodes={inner xsep=0.3333em, inner
ysep=0.3333em}]
    { \&  \&\&\& \& \& \& \&\& 6\\
     \& 5  \& \&\&\& \& \&\& \& 5\\
     \& 4  \& \&\& \& 4  \& \&\&\& 4\\
     \& 3  \& \&\& \& 3  \&\&\&\& 3\\
       1 \& \& 2 \&\&  1 \& \& 2 \& \&1 \& \& 2   \\ 
    };
\path[-> ] (n-5-1) edge [bend left=20] node[ scale=0.75, fill=white] [midway] {$a$} 
(n-5-3);
\path[-> ] (n-5-1) edge [bend right=20] node[ scale=0.75, fill=white] [midway] {$b$} 
(n-5-3);
\path[-> ] (n-5-5) edge [bend left=20] node[ scale=0.75, fill=white] [midway] {$a$} 
(n-5-7);
\path[-> ] (n-5-5) edge [bend right=20] node[ scale=0.75, fill=white] [midway] {$b$} 
(n-5-7);
\path[-> ] (n-5-9) edge [bend left=20] node[ scale=0.75, fill=white] [midway] {$a$} 
(n-5-11);
\path[-> ] (n-5-9) edge [bend right=20] node[ scale=0.75, fill=white] [midway] {$b$} 
(n-5-11);
\path[->] (n-5-3) edge (n-5-5);
\path[->] (n-5-7) edge (n-5-9);
\path[-> ] (n-5-3) edge [bend right=20] node[ scale=0.75, fill=white] [midway] {$b$} 
(n-4-2);
\path[-> ] (n-5-3) edge [bend right=30] node[ scale=0.75, fill=white] [midway] {$b$} 
(n-3-2);
\path[-> ] (n-5-3) edge [bend right=40] node[ scale=0.75, fill=white] [midway] {$b$} 
(n-2-2);
\path[-> ] (n-4-2) edge [bend right=20] node[ scale=0.75, fill=white] [midway] {$a$} 
(n-5-1);
\path[-> ] (n-3-2) edge [bend right=30] node[ scale=0.75, fill=white] [midway] {$a$} 
(n-5-1);
\path[-> ] (n-2-2) edge [bend right=40] node[ scale=0.75, fill=white] [midway] {$a$} 
(n-5-1);
\path[-> ] (n-5-7) edge [bend right=20] node[ scale=0.75, fill=white] [midway] {$b$} 
(n-4-6);
\path[-> ] (n-5-7) edge [bend right=30] node[ scale=0.75, fill=white] [midway] {$b$} 
(n-3-6);
\path[-> ] (n-4-6) edge [bend right=20] node[ scale=0.75, fill=white] [midway] {$a$} 
(n-5-5);
\path[-> ] (n-3-6) edge [bend right=30] node[ scale=0.75, fill=white] [midway] {$a$} 
(n-5-5);
\path[-> ] (n-5-11) edge [bend right=20] node[ scale=0.75, fill=white] [midway] {$b$} 
(n-4-10);
\path[-> ] (n-5-11) edge [bend right=30] node[ scale=0.75, fill=white] [midway] {$b$} 
(n-3-10);
\path[-> ] (n-5-11) edge [bend right=40] node[ scale=0.75, fill=white] [midway] {$b$} 
(n-2-10);
\path[-> ] (n-5-11) edge [bend right=50] node[ scale=0.75, fill=white] [midway] {$b$} 
(n-1-10);
\path[-> ] (n-4-10) edge [bend right=20] node[ scale=0.75, fill=white] [midway] {$a$} 
(n-5-9);
\path[-> ] (n-3-10) edge [bend right=30] node[ scale=0.75, fill=white] [midway] {$a$} 
(n-5-9);
\path[-> ] (n-2-10) edge [bend right=40] node[ scale=0.75, fill=white] [midway] {$a$} 
(n-5-9);
\path[-> ] (n-1-10) edge [bend right=50] node[ scale=0.75, fill=white] [midway] {$a$} 
(n-5-9);
\end{tikzpicture}
\end{align*}
\end{ex}


In order to see that the triangular matrix algebras $B_\mathbf{l}$ are again quasi-hereditary, we include the following well-known lemma, see e.g. Zhu \cite[Theorem 3.1.]{Zhutriangular}.

\begin{lem}\label{L:QHtria}
Let $B'$ and $B''$ be quasi-hereditary algebras with respect to linear orders $1'<2'<\cdots <n'$ and $1''<2''< \cdots<m''$, respectively. Let $_{B''}M_{B'}$ be a $B''$-$B'$-bimodule and let 
\begin{align*}
B=\begin{pmatrix}
B'' & _{B''}M_{B'} \\ 0 & B'
\end{pmatrix}.
\end{align*}
If the left $B''$-module $_{B''}M$ is filtered by standard $B''$-modules, then the linear order \[1'<2'<\cdots <n'<1''<2''< \cdots<m''\] defines a quasi-hereditary structure on $B$ and the standard $B$-modules are given by
\begin{align*}
\begin{pmatrix}
0 & 0 \\ 0 & \Delta(1')
\end{pmatrix},
\ldots, 
\begin{pmatrix}
0 & 0 \\ 0 & \Delta(n')
\end{pmatrix},
\begin{pmatrix}
\Delta(1'') & 0 \\ 0 & 0
\end{pmatrix},
\ldots,
\begin{pmatrix}
\Delta(m'') & 0 \\ 0 & 0
\end{pmatrix},
\end{align*}
where $\Delta(i')$ and $\Delta(j'')$ denote the standard modules for $B'$ and $B''$, respectively.
\end{lem}

\begin{prop} \label{P:MainJH}
Let $\mathbf{l}=(l_1, \ldots, l_m)$ be a sequence positive integers. Then $B_\mathbf{l}$ is a quasi-hereditary algebra with at least $2^m$ derived composition series of pairwise different length. The derived composition factors are given by $\cd^*(k)$,  $\cd^*(G_{l_1})$, \ldots, $\cd^*(G_{l_m})$. 
\end{prop}
\begin{proof}
A combination of Lemma \ref{L:QHtria}, Corollary \ref{QH} and an inductive argument shows that $B_\mathbf{l}$ is quasi-hereditary. Lemma \ref{L:QHtria} applies since the bimodules in our construction \eqref{E:glue} are projective left $B''$-modules, which are filtered by standard $B''$-modules by definition of a quasi-hereditary algebra. The statement about derived composition series follows from Proposition \ref{CEx} in conjunction with Corollary \ref{C:Recoll}.
\end{proof}

\section{Derived equivalence classification of certain gentle algebras} \label{S:BobMali}

\begin{defn}\label{D:Gentle}
Let $Q$ be a finite quiver with set of arrows $Q_1$.
A \emph{gentle algebra} is a finite dimensional $k$-algebra $kQ/I$ such that:
\begin{itemize}
\item[(G1)]  At any vertex, there are at most two incoming and at most two outgoing arrows.
\item[(G2)] $I$ is a two-sided admissible ideal, which is generated by paths of length two.
\item[(G3)] For each arrow $\beta \in Q_{1}$,  there is at most one arrow $\alpha \in Q_{1}$ such that 
$0 \neq \alpha\beta \in I$ and at most one arrow $\gamma \in Q_{1}$ such that 
$0 \neq \beta \gamma \in I.$
\item[(G4)] For each arrow $\beta\in Q_{1}$, there is at most one arrow $\alpha \in Q_{1}$ such that 
$\alpha\beta \notin I$ and at most one arrow $\gamma \in Q_{1}$ such that 
$\beta \gamma \notin I.$
\end{itemize}
\end{defn}

\begin{rem}
It is well-known that gentle algebras can also be characterised as those finite dimensional algebras with special biserial repetitive algebras, see for example Schr{\"o}er \cite[Section 4]{Schroer} and also Ringel \cite{Ringel}.
\end{rem}

\begin{prop} \label{RankEulerform}
The algebras $A_1$ and $A_2$ from Proposition \ref{P:Notequiv} are the only gentle algebras of finite global dimension with three vertices and four arrows such that the rank of the symmetrized Euler form is $1$.
\end{prop}
\begin{proof}
We check this case by case. For this we give a list of all gentle algebras $kQ/I$ of finite global dimension with three vertices and four arrows - in particular, the quivers $Q$ do not contain loops, see e.g. \cite{Igusa}. Therefore, there have to be two vertices which are connected by at least two arrows.

\medskip

\noindent
{\bf Case 1: Kronecker.} Assume that $Q$ contains a subquiver of the form 
\[
\begin{xy}
\SelectTips{cm}{10}
\xymatrix{1 \ar@/^/[r] \ar@/_/[r]  & 2.   }
\end{xy} 
\]

There are the following four possibilities to extend this to a connected gentle quiver $Q$ with three vertices, four arrows and no loops. Namely,
\begin{align}
Q_1, Q_2, Q_3:=\begin{array}{c}{\begin{tikzpicture}[description/.style={fill=white,inner sep=0pt}]
    \matrix (n) [matrix of math nodes, ampersand replacement=\&, row sep=1em,
                 column sep=2.25em, text height=1.5ex, text depth=0.25ex,
                 inner sep=0pt, nodes={inner xsep=0.3333em, inner
ysep=0.3333em}]
    {   1 \& \& \& \& 2 \& \& 3 \\
    };
\draw[->] (n-1-1) edge [bend left=10] node[scale=0.75, fill=white] [midway] {$c$} (n-1-5);
\draw[->] (n-1-1) edge [bend right=10] node[ scale=0.75, fill=white] [midway] {$b$}  (n-1-5);
\draw[<-] (n-1-7) edge node[ scale=0.75, , fill=white] [midway] {$d$}  (n-1-5);
\draw[<-] (n-1-1) edge [bend left=30] node[scale=0.75, fill=white] [midway] {$a$} (n-1-5);
\end{tikzpicture}Ê}\end{array}, Q^{\rm op}_3.
\end{align}

Up to algebra isomorphism the algebra $A_1$ from Proposition \ref{P:Notequiv} above is the unique gentle algebra with underlying quiver $Q_1$. Moreover, $A_3=kQ_3/(ca, ab, dc)$ is (up to algebra isomorphism) the unique gentle algebra with underlying quiver $Q_3$ and dually $Q^{\rm op}_3$ gives rise to $A_3^{\rm op}$ (again unique up to isomorphism). 
There are two gentle algebra structures (of finite global dimension) on $Q_2$ up to isomorphism ($A_2$ and $A'_2:=kQ_2/I'_2$ from \eqref{E:CJHalgebra}, which is the starting point for our examples in Section \ref{S:JH}). The rank of the symmetrized Euler form for all  algebras (except for $A_1$ and $A_2$) arising in this way is $2$. 

\medskip

\noindent
{\bf Case 2: Two-cycle.}\label{subs:Twocycle}
Assume that $Q$ contains a subquiver of the form 
\[
\begin{xy}
\SelectTips{cm}{10}
\xymatrix{1 \ar@/^/[rr]|{\displaystyle \,\, x  \,\, }   &&  2. \ar@/^/[ll]|{\displaystyle \,\, y  \,\, }    }
\end{xy} 
\]

In addition to $Q_3, Q_3^{\rm op}$ there is the following family of quivers
\[
\begin{tikzpicture}[description/.style={fill=white,inner sep=2pt}]
    \matrix (n) [matrix of math nodes, row sep=1em, ampersand replacement=\&,
                 column sep=1.3em, text height=1.5ex, text depth=0.25ex,
                 inner sep=0pt, nodes={inner xsep=0.3333em, inner
ysep=0.3333em}]
    { \& 3 \\
       1 \& \& 2   \\ 
    };
\path[-> ] (n-2-1) edge [bend left=20] node[ scale=0.75, fill=white] [midway] {$x$} 
(n-2-3);
\path[<- ] (n-2-1) edge [bend right=20] node[ scale=0.75, fill=white] [midway] {$y$} 
(n-2-3);
\path[-] (n-2-3) edge node[ scale=0.75, xshift=9pt, yshift=3pt] [midway] {$z_2$} 
(n-1-2);
\path[-] (n-1-2) edge node[ scale=0.75, xshift=-9pt, yshift=3pt] [midway] {$z_1$} 
(n-2-1);

\end{tikzpicture}
\]


where the edges $z_1$ and $z_2$ can have an arbitrary orientation. One can check that the rank of the symmetrized Euler-form is $2$ for all of these algebras.

There is no way to define a finite dimensional gentle algebra of finite global dimension on the following quiver:

\[
\begin{xy}
\SelectTips{cm}{10}
\xymatrix{1 \ar@/^/[r]   & 2 \ar@/^/[l] \ar@/^/[r]  & 3  \ar@/^/[l]   }
\end{xy} 
\]

Summing up, $A_1$ and $A_2$ are the only gentle algebras of finite global dimension with three vertices and four arrows such that the rank of the symmetrized Euler form is $1$.
\end{proof}

\begin{rem}\label{R:44} One can compute that all gentle algebras of finite global dimension with three vertices \& four arrows have AG-invariant $[2, 4]$. This can also be deduced from \cite[Lemma 3.1]{BM} as all these gentle algebras are degenerate in the sense of \cite{BM}, see e.g.~the proof of Corollary \ref{C:DerivedEq}. 

Moreover, the AG-invariant is a complete derived invariant for gentle algebras with at most two vertices. Indeed this follows from the classification of  Bessenrodt \& Holm \cite[Example 3.7]{BH} in combination with Ladkani \cite{Ladkani} and the definition of the AG-invariant, which detects oriented cycles with full zero relations. In particular, Proposition \ref{P:Notequiv} provides a minimal\footnote{We thank Sefi Ladkani for pointing this out.} example showing that the AG-invariant is not sufficient to distinguish derived categories of gentle algebras, see \cite{BM} \& \cite{Amiot} for further examples.
\end{rem}

\begin{defn}\label{D:derivedisolated}
We call a noetherian ring $A$ \emph{derived unique}\footnote{We thank Michael Wemyss for suggesting this terminology.} if every ring $B$ which is derived equivalent to $A$ is already Morita equivalent to $A$. In other words, the derived equivalence class and the Morita equivalence class of $A$ coincide. 
\end{defn} 

\begin{rem} \label{R:Derunique}
\begin{enumerate}
\item Examples of derived unique algebras include commutative algebras \cite{RouquierZimmermann}, local algebras \cite{Zimmermann2},  path algebras of $n$-Kronecker quivers, preprojective algebras of Dynkin (\cite{AiharaMizuno}) and of extended Dynkin type (\cite{IyamaWemyssNCBondalOrlov})...
\item In algebraic geometry, Bondal \& Orlov \cite{BO} showed that the derived category $\cd^b(\coh X)$ of a smooth projective variety $X$ with ample canonical or anticanonical bundle determines $X$. It would be interesting to look for analogous results for derived unique algebras. 
\end{enumerate}
\end{rem}

We finish with a complete description of derived equivalence classes of gentle algebras of finite global dimension with three vertices and four arrows, this also follows from Bobi{\'n}ski \cite{Bobinski}.

\begin{cor} \label{C:DerivedEq} \label{C:derivedisolated}
There are three derived equivalence classes of gentle algebras of finite global dimension with three vertices and four arrows 
\begin{align*}
\{A_1\}, \{A_2\}, \{ \text{ algebras derived equivalent to } A'_2\}.
\end{align*} 
In particular, the algebras $A_1$ and $A_2$ are derived unique. 
\end{cor}
\begin{proof}
We first show that  $A_1$ \& $A_2$ are derived unique. Only this part is used in Section \ref{S:LY}.

Let $\cd^b(A_i) \cong \cd^b(B)$ be a triangle equivalence, with $i=1$ or $2$. 
The main result of Schr{\"o}er \& Zimmermann \cite{SZ} shows that $B$ is Morita equivalent to a gentle algebra $C=kQ/I$. Since the rank of the Grothendieck group is a derived invariant $Q$ has three vertices. The number of arrows is a derived invariant by work of Avella-Alaminos \& Gei{\ss} \cite{AG}. So $Q$ has four arrows. The rank of the symmetrized Euler form (cf. \cite[Proposition, p. 101]{Happelbook}) and finiteness of global dimension are invariant under derived equivalences, therefore $C \cong A_1$ or $A_2$ by Proposition \ref{RankEulerform}. By Proposition \ref{P:Notequiv}, $A_1$ and $A_2$ are not derived equivalent. Summing up, $A_1$ and $A_2$ are derived unique.

In order to complete the description of the derived equivalence classes, it remains to show that all gentle algebras of finite global dimension and with three vertices and four arrows are derived equivalent to $A'_2$. 
This follows from Bobi{\'n}ski \& Malicki \cite[Theorem 2]{BM} once we show that all these algebras are \emph{degenerate} in their sense. This can either be checked by direct computation using the list in Proposition \ref{RankEulerform} or one can proceed as follows. By definition a gentle two-cycle algebra (all our algebras are of this form) is either degenerate or non-degenerate. Moreover, this property is invariant under derived equivalences. Bobi{\'n}ski \& Malicki \cite[Theorem 1]{BM} give a list of representatives of derived equivalence classes of non-degenerate gentle two-cycle algebras. One can check that there is no representative which has three vertices and finite global dimension. Finiteness of global dimension is invariant under derived equivalences. Since, by assumption, our algebras are of finite global dimension, they cannot be derived equivalent to a non-degenerate algebra and therefore are indeed degenerate. This completes the proof.
\end{proof}

\subsection*{Derived equivalences and the AG-invariant}
This subsection contains background material on work of Avella-Alaminos \& Gei{\ss} \cite{AG}, who describe the structure of certain characteristic components of the Auslander-Reiten quiver of the derived category of a gentle algebra (of finite global dimension), leading to the definition of a derived invariant (called AG-invariant).

 Building on this, we modify given derived equivalences such that they identify certain prescribed objects (Lemma \ref{L:Key}). In special cases, this yields derived equivalences between corner algebras, by passing to triangulated quotient categories (Corollary \ref{C:Kick}). This is used to simplify a conjecture of Bobi{\'n}ski \& Malicki \cite{BM}
in the next subsection (Corollary \ref{C:RedBobMali}) and also in the proof of Proposition \ref{P:Notequiv} in Section \ref{S:Proof}. 

We start with a general lemma, which we apply to gentle algebras in Corollary \ref{C:Kick}.

\begin{lem}\label{L:Kickout}
Let $A$ and $B$ be finite dimensional algebras such that there is a triangle equivalence $\Phi \colon \cd^b(A) \to \cd^b(B)$.
Let $S_i$ be an exceptional simple $A$-module and assume that there exists a simple $B$-module $S_j$ and an autoequivalence $\Psi$ of $\cd^b(B)$ such that $\Phi(S_i) \cong \Psi(S_j)$. Then there is a triangle equivalence
\begin{align}
\cd^b((1-e_i) A(1 - e_i)) \cong  \cd^b((1-e_j) B(1 - e_j))
\end{align}
where $e_i \in A$ and $e_j \in B$ are the idempotents corresponding to $S_i$ and $S_j$, respectively.

\end{lem}
\begin{proof}
Our assumptions imply that the triangle equivalence $\Psi^{-1} \circ \Phi\colon \cd^b(A) \to \cd^b(B)$ sends $S_i$ to $S_j$. This gives an induced equivalence between Verdier quotient categories $\cd^b(A)/\thick(S_i) \cong \cd^b(B)/\thick(S_j)$. The proof of Proposition \ref{P:AdjTriple} shows that the functor $\cd^b(A) \to \cd^b((1-e_i) A (1-e_i))$ induced by multiplication with $(1-e_i)$ is a quotient functor. Its kernel is $\thick(A/A(1-e_i)A-\mod)$, which equals  $\thick(S_i)$ since $S_i$ is exceptional by assumption. In particular, $\cd^b(A)/\thick(S_i) \cong \cd^b((1-e_i) A (1-e_i))$ and we have an analogous result for $B$ since $S_j \cong \Psi^{-1}(\Phi (S_i))$ is exceptional. This completes the argument. 
\end{proof}

\subsubsection*{Characteristic components} Let $\Lambda=kQ/I$ be a gentle algebra of finite global dimension.

\begin{defn}\label{D:Char}
A \emph{characteristic component} (CC) of $\Lambda$ is a connected component $\cc$ of the Auslander-Reiten quiver of $\cd^b(\Lambda)$ such that
\begin{itemize}
\item[(i)] $\cc$ has a boundary, i.e. $\cc$ contains an Auslander-Reiten triangle $\tau X \to Y  \to X \to \nu X$ such that $Y$ is indecomposable.
\item[(ii)] indecomposable objects in $\cc$ are mapped to string modules over the repetitive algebra $\widehat{\Lambda}$ under Happel's equivalence $\cd^b(\Lambda) \cong \widehat{\Lambda}-\ul{\mod}$ \cite{Happelbook}.
\end{itemize}
\end{defn} 

\noindent
 Building on \cite{GeissPena} and \cite{ButlerRingel}, characteristic components are classified in \cite{AG}.

\begin{prop}\label{P:CCC}
The following translation quivers occur as characteristic components
\begin{itemize}
\item[(i)] $\ZZ\AA_n$ for some $n \in \ZZ_{\geq 1}$.
\item[(ii)] $\ZZ\AA_\infty$.
\item[(iii)] $\ZZ\AA_\infty/\tau^r$ for some $r \in \ZZ_{\geq 1}$. 
\end{itemize}
Conversely, every AR-component $\cc \neq \ZZ\AA_\infty/\tau$ of $\cd^b(\Lambda)$ which appears in (i) - (iii) is a CC.
\end{prop}

In particular, derived autoequivalences act transitively on the boundary of a CC.

\begin{cor}\label{C:Auto}
If $X$ and $Y$ are objects in the boundary of a CC $\cc$ of $\Lambda$, then there exists a derived autoequivalence $\psi$ of $\cd^b(\Lambda)$ such that $\psi(X)\cong Y$.
\end{cor}
\begin{proof}
If $\cc$ is of type $\ZZ\AA_\infty$ or $\ZZ\AA_\infty/\tau^r$, then there is a unique boundary component. By definition the AR-translation $\tau$ acts transitively on it. Therefore, we can take $\psi=\tau^t$ for some $t \in \ZZ$.

If $\cc$ is of type $\ZZ\AA_n$, then there are two boundary components. The AR-translation acts transitively on each of these components and the shift functor $[1]$ identifies the two components. So either $\psi=\tau^t [1]$ or $\psi=\tau^t$  for some $t \in \ZZ$ will identify $X$ and $Y$ in this case.  
\end{proof}

We call a CC $\cc$ \emph{homogeneous} if $\cc=\ZZ\AA_\infty/\tau$ and  \emph{non-homogeneous} otherwise.

\begin{cor}\label{C:AutoCC}
Let $\Lambda, \Lambda'$ be gentle algebras of finite global dimension. If $\psi\colon \cd^b(\Lambda) \to \cd^b(\Lambda')$ is a derived equivalence and $\cc$ is a non-homogeneous CC of $\Lambda$, then $\psi(\cc)$ is a non-homogeneous CC of $\Lambda'$. Moreover, $\psi$ maps objects in the boundary of $\cc$ to the boundary of $\psi(\cc)$. 
\end{cor}
\begin{proof}
Since $\psi$ is a triangle equivalence it commutes with Serre-functors $\nu$ and shift functors $[1]$. In particular, it commutes with the AR-translations $\tau=\nu \circ [-1]$. Therefore, $\psi$ maps AR-components to AR-components and boundary objects to boundary objects. Finally, by assumption $\cc \neq \ZZ\AA_\infty/\tau$ appears in (i) - (iii) of Proposition \ref{P:CCC}, thus $\psi(\cc) \neq \ZZ\AA_\infty/\tau$ occurs in this list as well. So $\psi(\cc)$ is a CC by Proposition \ref{P:CCC}. 
\end{proof}

\begin{rem}
In general, derived autoequivalences can identify homogeneous CCs with non-characteristic components. For example, this happens for the Kronecker quiver.
\end{rem}

\subsubsection*{The AG-invariant} One can check\footnote{By Corollary \ref{C:AutoCC}, it remains to check the homogeneous CCs.} that the shift functor (on the stable category of the repetitive algebra this is given by the inverse syzygy functor) acts on characteristic components. Avella-Alaminos \& Gei{\ss} \cite{AG} describe the orbits under this action: there are finitely many orbits and the corresponding triangulated subcategories are fractionally Calabi-Yau. More precisely, there is an algorithm \cite[Section 3]{AG} with output a formal sum of pairs of integers $n_i, m_i \in \ZZ_{\geq 0}$ 
\begin{align}
\phi_\Lambda=[n_1, m_1] + \ldots + [n_t, m_t]
\end{align}
obtained from counting certain walks in the gentle quiver with relations defining $\Lambda$\footnote{Repetitions $[n_i, m_i]= [n_j, m_j]$ for $i \neq j$ may occur. This notation differs from the one used in \cite{AG} but contains the same information.}. The next result gives a categorical interpretation of this combinatorial information, see \cite[Theorem 16]{AG}.

\begin{prop}\label{P:CatInter}
The summands $[n_i, m_i]$ of $\phi_\Lambda$ are in bijection with the $[1]$-orbits of the CCs of $\Lambda$. Moreover, every object $X$ in a component of the orbit corresponding to $[n_i, m_i]$ satisfies
\begin{align}\label{E:frac}
\nu^{n_i}X \cong X[m_i]
\end{align} 
where $\nu$ denotes the Nakayama=Serre functor of $\cd^b(\Lambda)$. In other words, $X$ is $\frac{m_i}{n_i}$-fractionally CY.
\end{prop}

\begin{cor} \label{L:homo}
If there exists a homogeneous CC of $\Lambda$, then $[1, 1]$ is a summand of $\phi_\Lambda$. 
\end{cor}
\begin{proof}
Since $\tau=\nu \circ [-1]$ and by definition $\tau X \cong X$ for objects $X$ in homogeneous CCs, the statement follows from Proposition \ref{P:CatInter}.
\end{proof}

\noindent
Using this interpretation of $\phi_\Lambda$, Avella-Alaminos \& Gei{\ss} obtain their main result \cite[Theorem A]{AG}.
\begin{thm}
Let $A$ and $B$ be derived equivalent gentle algebras, then $\phi_A=\phi_B$. 
\end{thm}

\noindent
Therefore, we call $\phi_\Lambda$ the \emph{AG-invariant} of $\Lambda$.

\subsubsection*{Boundaries of CCs, permitted threads and modifications of derived equivalences}

The objects in the boundary of CCs are classified, see e.g. \cite[Section 2.3]{AG}. For our purposes it is enough to understand which $\Lambda$-modules are in the boundary. We need the following definition.

\begin{defn} \label{D:thread}
Let $A=kQ/I$ be a gentle algebra. A \emph{non-trivial
 permitted thread} of $A$ is a maximal path $p$ in $(Q, I)$, i.e. $p$ is not contained in $I$ but any path in $Q$ with subpath $p$ is contained in $I$. A \emph{trivial permitted thread} is a trivial path $1_v$ where $v$ is a vertex in $Q$ such that 
 \begin{itemize}
 \item[(a)] there is at most one arrow $\alpha$ ending in $v$.
 \item[(b)] there is at most one arrow $\beta$ starting in $v$.
 \item[(c)] if both arrows $\alpha$ and $\beta$ in (a) and (b) exist, then $\beta\alpha \notin I$.
 \end{itemize}
 A \emph{permitted thread} is a trivial or non-trivial permitted thread.
\end{defn}

Building on \cite{ButlerRingel, WW}, Avella-Alaminos \& Gei{\ss} \cite{AG} show the following result, cf. also \cite{GeissPena}.

\begin{prop}\label{P:Boundary}
Let $p$ be a permitted thread of $\Lambda$. Then the corresponding string module $M(p)$ is contained in the boundary of a CC of $\Lambda$.
\end{prop}

For simplicity, in the rest of this subsection, we restrict ourselves to the case where the AG-invariant has only one summand. More generally, one could work with an orbit of CCs which corresponds to a summand $[n_i, m_i]$ such that $[n_i, m_i] \neq [n_j, m_j]$ for all $j\neq i$.

\begin{Setup}\label{S:DE}
Let $A$ and $B$ be derived equivalent gentle algebras such that $\phi_B=[n, m]$.
\end{Setup}

\begin{rem} \label{R:gldim}
It follows that $B=kQ/I$ (and therefore also $A$) has finite global dimension. 

Indeed if $\gldim B=\infty$ then there is a cyclic path with full relations in $(Q, I)$, see e.g. \cite{KalckSing}. By \cite[Remark 6]{AG}, this shows that $\phi_B$ has a summand of the form $[0, l(c)]$, where $l(c)$ is the length of $c$. By our assumption this would imply $\phi_B=[0, l(c)]$, but the algorithm producing $\phi_B$ shows that there is always at least one summand $[n_i, m_i]$ with $n_i \neq 0$. Contradiction. So $B$ has finite global dimension. In particular, Happel's equivalence $\cd^b(B) \cong \widehat{B}-\ul{\mod}$ holds and we can use all the statements listed before Setup \ref{S:DE}.
\end{rem}

\begin{lem}\label{L:Noteq}
In the notation of Setup \ref{S:DE}, we have $n\neq m$. \end{lem}
\begin{proof}
Assume that $\phi_B=[n, n]$. It follows from \cite[Lemma 3.2]{BobinskiBuan} that the underlying quiver $Q$ of $B$ has $n$ vertices and $n$ arrows. Since $Q$ is connected (indeed otherwise $\phi_B$ has at least two summands), the unoriented graph underlying $Q$ is a cycle of length $n$. It follows from \cite[Section 7]{AG} that $\phi_B$ has exactly two summands. Contradiction. So we see that indeed $n\neq m$.
\end{proof}

In the situation of Setup \ref{S:DE}, the derived equivalence can be adapted such that it identifies given objects $X \in \cd^b(A)$, $Y \in \cd^b(B)$ contained in boundaries of CCs (cf. also \cite[Corollary 1.4]{Bobinski}).
\begin{lem}\label{L:Key}
In the notation of Setup \ref{S:DE}, let $p$ and $q$ be permitted threads of $A$  and $B$ respectively. Then there exists a derived equivalence $\psi\colon \cd^b(A) \to \cd^b(B)$ such that $\psi(M(p))\cong M(q)$.
\end{lem}
\begin{proof}
Let $\gamma \colon \cd^b(A) \to \cd^b(B)$ be the given derived equivalence. Lemma \ref{L:Noteq} and Corollary \ref{L:homo} imply that all CCs of $A$ and $B$ are non-homogeneous. In combination with Proposition \ref{P:Boundary} this shows that $M(p)$ is contained in the boundary of a non-homogeneous CC of $A$. So $\gamma(M(p))$ is contained in the boundary of a CC of $B$ by Corollary \ref{C:AutoCC}. Since $\phi_B=[n, m]$ has a single summand, there is a unique orbit of CCs of $B$ under the shift functor by Proposition \ref{P:CatInter}. In particular there exists an integer $s$ such that $\gamma(M(p))[s]$ and $M(q)$ are contained in the same (non-homogeneous) CC of $B$. Corollary \ref{C:Auto} yields an autoequivalence $\alpha$ of $\cd^b(B)$ such that $\alpha\gamma(M(p))[s] \cong M(q)$. Thus $\psi:=\alpha \circ \gamma \circ [s]$ is an autoequivalence with the desired properties.
\end{proof}

\noindent
Combining Lemmas \ref{L:Key} and \ref{L:Kickout} yields the following corollary, which is used in Proposition \ref{P:Prep}.

\begin{cor}\label{C:Kick}
In the notation of Lemma \ref{L:Key}, assume that $p=1_v$ and $q=1_w$ are trivial permitted threads such that the corresponding simple string modules $M(p)=S_v$ and $M(q)=S_w$ are exceptional. 

Then there exists a derived equivalence $\cd^b((1-e_v)A(1-e_v)) \cong \cd^b((1-e_w)B(1-e_w))$, where $e_v$ and $e_w$ are the idempotents corresponding to the vertices $v$ and $w$. 
\end{cor}

The next result plays a key role in our proof of Proposition \ref{P:Notequiv}, see Section \ref{S:Proof}.

\begin{cor}\label{C:Key}
In the notation of Lemma \ref{L:Key}, let $q$ be a permitted thread such that $M(q)$ is a projective $B$-module. Then there exists a tilting object $T$ in $\cd^b(A)$ which contains $M(p)$ as a direct summand and satisfies $\End_A(T) \cong B^{\rm op}$.
\end{cor}
\begin{proof}
Lemma \ref{L:Key} shows that there is a derived equivalence $\psi\colon \cd^b(A) \to \cd^b(B)$ such that $\psi(M(p)) \cong M(q)$. Rickard's derived Morita theory \cite{Rickard} shows that there is a  tilting object $T  \in \cd^b(A)$ such that $\End_A(T) \cong B^{\rm op}$ and $\psi(T) \cong B$. Since $M(q)$ is projective it is a direct summand of $B$. Therefore $\psi^{-1}(M(q)) \cong M(p)$ is a direct summand of $T$ completing the proof.
\end{proof}

\subsection*{On the conjecture of Bobi{\'n}ski \& Malicki}

Following \cite{BM}, we define the following two families of gentle two-cycle algebras:

\[
\begin{tikzpicture}[description/.style={fill=white,inner sep=2pt}]
    \matrix (n) [matrix of math nodes, row sep=1em, ampersand replacement=\&,
                 column sep=1.3em, text height=1.5ex, text depth=0.25ex,
                 inner sep=0pt, nodes={inner xsep=0.3333em, inner
ysep=0.3333em}]
    { \&\& \&\& 3 \& 4 \& \cdots \& p+1 \\
       \mathbb{Q}_p:= \&1 \& \& 2  \&  \&\& \&\& p+2  \\ 
           };

\path[-> ] (n-2-2) edge [bend right=20] node[ scale=0.75, fill=white] [midway] {$b$} 
(n-2-4);
\path[-> ] (n-2-4) edge node[ scale=0.75, fill=white] [midway] {$b$} 
(n-2-9);

\path[-> ] (n-2-2) edge [bend left=20] node[ scale=0.75, fill=white] [midway] {$a$} 
(n-2-4);
\path[->] (n-2-4) edge node[ scale=0.75, 
xshift=-8pt, yshift=4pt
] [midway] {$a$} 
(n-1-5);
\path[->] (n-1-5) edge  
(n-1-6);
\path[->] (n-1-6) edge  
(n-1-7);
\path[->] (n-1-7) edge  
(n-1-8);
\path[->] (n-1-8) edge  
(n-2-9);

\draw (n-1-1) node[yshift=-59pt, xshift=-24pt] {$\mathbb{I}_{p}:=( a^2, b^2)$ };

\end{tikzpicture}
\]


defines algebras $A(p):=k\mathbb{Q}_p/\mathbb{I}_p$ for every $p \in \mathbb{Z}_{\geq 1}$ and
 
\[
\begin{tikzpicture}[description/.style={fill=white,inner sep=2pt}]
    \matrix (n) [matrix of math nodes, row sep=1em, ampersand replacement=\&,
                 column sep=1.3em, text height=1.5ex, text depth=0.25ex,
                 inner sep=0pt, nodes={inner xsep=0.3333em, inner
ysep=0.3333em}]
    { \& p+1  \&\& \cdots \&\& 3 \&\& 2\\ \\
       \mathbb{O}_p:= \& p+2 \&  \& \& \&  \& \& 1   \\ 
    };
\path[-> ] (n-3-2) edge [bend left=15] node[ scale=0.75, fill=white] [midway] {$\gamma$} 
(n-3-8);
\path[-> ] (n-3-2) edge [bend right=15] node[ scale=0.75, fill=white] [midway] {$\beta$} 
(n-3-8);
\path[->] (n-3-8) edge node[ scale=0.75, xshift=13pt] [midway] {$\alpha_{p+1}$} 
(n-1-8);
\path[->] (n-1-8) edge node[ scale=0.75, yshift=8pt] [midway] {$\alpha_{p}$} 
(n-1-6);
\path[->] (n-1-6) edge node[ scale=0.75, yshift=8pt] [midway] {$\alpha_{p-1}$} 
(n-1-4);
\path[->] (n-1-4) edge node[ scale=0.75, yshift=8pt] [midway] {$\alpha_2$} 
(n-1-2);
\path[->] (n-1-2) edge node[ scale=0.75, xshift=-9pt, yshift=3pt] [midway] {$\alpha_1$} 
(n-3-2);

\draw (n-1-1) node[yshift=-59pt, xshift=88pt] {$\mathbb{I}_{p}(r):=( \alpha_{p+1} \beta, \gamma \alpha_{1}, \alpha_{1} \alpha_{2}, \alpha_{2} \alpha_{3}, \ldots, \alpha_{r}\alpha_{r+1})$ };

\end{tikzpicture}
\]

defines algebras $B(p, r):=k\mathbb{O}_p/\mathbb{I}_p(r)$ for $p \in \ZZ_{\geq 0}$ and $r \in [0, p]$.
\begin{rem}
\begin{itemize}
\item[(a)] In \cite{BM}, the algebras $A(p)$ are denoted by $\Lambda'_0(p, 0)$ for $p \geq 1$ and the $B(p, r)$ are denoted by $\Lambda_0(p+1, r)$ for $p \geq 0$.
\item[(b)] The algebras from Proposition \ref{P:Notequiv} appear as special cases. Namely, $A_1=A(1)$ and $A_2=B(1, 1)$. Moreover, $A'_2=B(1, 0)$ is the algebra defined in \eqref{E:CJHalgebra}.
\end{itemize}
\end{rem}

The following conjecture of Bobi{\'n}ski \& Malicki \cite[Conjecture 1]{BM} states that different algebras from the families $A(p)$ and $B(p, r)$ are not derived equivalent. This was recently proved by Bobi{\'n}ski \cite{Bobinski} building on earlier work of Amiot \cite{Amiot}. 

\begin{Conjecture}[Bobi{\'n}ski \& Malicki] \label{C:BobMali} \quad \\
\noindent
{\rm (a)} $\cd^b(B(p, r)) \ncong \cd^b(A(p'))$ for all $p \in \ZZ_{\geq 0}, p' \in \ZZ_{\geq 1}, r \in [0, p]$.

\noindent
{\rm (b)} $\cd^b(B(p, r)) \cong \cd^b(B(p', r'))$ implies $p=p'$ and $r=r'$.

\end{Conjecture}

\begin{rem}\label{R:BM}
Since the rank of the Grothendieck group is a derived invariant, it follows that
 $\cd^b(B(p, r)) \cong \cd^b(B(p', r'))$ implies $p=p'$ in part (b) of the conjecture. By the same argument (a) holds if $p \neq p'$.

 Bobi{\'n}ski \& Malicki \cite[paragraph after Conjecture 1]{BM} check that part (a) holds for $r \equiv 0 \, ({\rm mod} \, 2)$ and also $\cd^b(B(p, r)) \cong \cd^b(B(p', r'))$ implies $r \equiv r' \, ({\rm mod} \, 2)$. Indeed the symmetrized Euler form of
 $A(p)$ has rank $p$, whereas the rank of the symmetrized Euler form of $B(p, r)$ is $p+1$ if $r$ is even and $p$ if $r$ is odd.
\end{rem}

\noindent
The next result will be used in Corollary \ref{C:RedBobMali} to show that one of the algebras in the conjecture can be assumed to be $B(p, p)$. Bobi{\'n}ski \cite[Corollary 2.2]{Bobinski} shows the converse statements of this proposition and combines them with Amiot's results \cite{Amiot} to prove Conjecture \ref{C:BobMali} in full generality.
\begin{prop} \label{P:Prep} Let $p \in \mathbb{Z}_{\geq 1}$ and $r, r' \in [0, p]$.

\begin{itemize}
\item[$(i)$] If $B(p, r)$ and  $B(p, r')$ are not derived equivalent, then $B(p+1, r)$ and  $B(p+1, r')$ are not derived equivalent.
\item[$(ii)$] If $B(p, r)$ is not derived equivalent to $A(p)$, then  $B(p+1, r)$ and $A(p+1)$ are not derived equivalent. 
\end{itemize}
\end{prop}

\begin{proof}
We can apply Corollary \ref{C:Kick} to the simple modules $S_2$ in $\cd^b(B(p+1, r))$ respectively $S_3$ in $\cd^b(A(p+1))$ to show the contraposition of $(i)$ and $(ii)$, respectively. 

Indeed one can check that both simple modules are exceptional and defined by trivial permitted threads. Moreover, it follows by direct calculation or using \cite[Lemma 3.1]{BM} that the AG-invariants of $B(p+1, r)$ for all $r$ and $A(p+1)$ are $[p+2, p+4]$.  Since $(1-e_2) B(p+1, r) (1-e_2) \cong B(p, r)$ for all $r<p+1$ and 
$(1-e_3) A(p+1) (1-e_3) \cong A(p)$, Corollary \ref{C:Kick} proves the contraposition of $(i)$ respectively $(ii)$.
\end{proof}

Using Proposition \ref{P:Prep} iteratively, we can now reduce the conjecture as follows.

\begin{cor} \label{C:RedBobMali}
It is enough to check the following cases of Conjecture \ref{C:BobMali}:
\begin{itemize}
\item[$(i)$] $B(p, p)$ and $B(p, r)$ are not derived equivalent for all $p \geq 0$ and $r<p$. 
\item[$(ii)$] $B(p, p)$ and $A(p)$ are not derived equivalent for all $p \geq 1$.
\end{itemize}
\end{cor}

\begin{rem}\label{R:Extend}
It seems that the approach used in the proof of Proposition \ref{P:Notequiv} could be extended to deal with (some of) the cases of Conjecture \ref{C:BobMali} remaining after the reduction of Corollary \ref{C:RedBobMali}, we refer to Remark \ref{R:Extension} for more details.
\end{rem}

Combining Proposition \ref{P:Notequiv} with Proposition \ref{P:Prep} $(ii)$, we obtain:
\begin{cor} \label{C:Specialcase}
$\cd^b(B(p, 1)) \ncong \cd^b(A(p))$ for all $p>0$.
\end{cor}

\section{Exceptional sequences and quasi-hereditary algebras \\ -- a negative answer to a question of Liu \& Yang} \label{S:LY}

The aim of this section is to give a negative answer to Question \ref{Q:CharQH} of Liu \& Yang (cf. \cite[Question 1.1]{derivedsimple}), which we restate below for the convenience of the reader. We refer to Section \ref{S:RecollementsETC} for unexplained terminology.

\begin{Question}\label{Q:LY} Let $k$ be a field and let $A$ be a finite-dimensional $k$-algebra. Assume that the derived category of finitely generated left $A$-modules $\cd^b(A-\mod)$ admits a composition series with simple factors being finite-dimensional division $k$-algebras. Is $A$ derived equivalent to a quasi-hereditary algebra?
\end{Question}

\begin{rem}
Liu \& Yang actually state the question with $\cd^b(A-\mod)$ replaced by the unbounded derived category of all $A$-modules $\cd(A-\Mod)$. However, it follows from \cite{AKL} and \cite{AKLY} that any $\cd^b(A-\mod)$ composition series lifts to a $\cd(A-\Mod)$-composition series. In particular, a negative answer to Question \ref{Q:LY} yields a negative answer to \cite[Question 1.1]{derivedsimple}. 
\end{rem}


\begin{lem}\label{L:Strat}
The algebra $A_2$ admits a composition series by derived categories $\cd^b(k)$. 
\end{lem}
\begin{proof}
The simple $A:=A_2$-module $S_2$ has the following projective resolution $0 \to P_1 \to P_3 \to P_2$. In particular, $S_2 \cong A/AeA$ is an exceptional $A$-module, where $e=e_1+e_3$. This shows that the canonical functor $\cd^b(k)=\cd^b(A/AeA) \to \cd^b(A)$ is fully faithful and its image is $\thick(A/AeA-\mod)$. Therefore, there is a recollement $(\cd^b(k), \cd^b(A), \cd^b(eAe))$ by Remark \ref{R:Recollidemp}. But $eAe \cong kQ$, where $Q$ is the Kronecker quiver, which admits a recollement $(\cd^b(k), \cd^b(kQ), \cd^b(k))$ induced by the exceptional sequence given by the simple $kQ$-modules (Lemma \ref{L:ExcSeqCompSeries}). This finishes the proof.
\end{proof}

\begin{rem}
We note that all categories involved in the first recollement $(\cd^b(k), \cd^b(A), \cd^b(eAe))$ in the proof of Lemma \ref{L:Strat} have Serre functors. It is well-known that this can be used to obtain a recollement  $(\cd^b(eAe), \cd^b(A), \cd^b(k))$ in the opposite direction, see e.g. \cite[Proposition 3.7]{AKLY}. However, combining the proof of Lemma \ref{L:Notquasi} with Corollary \ref{C:derivedisolated} this recollement is not equivalent to a recollement of the form $(\cd^b(B/BeB), \cd^b(B), \cd^b(eBe))$ where $B$ is a finite dimensional algebra and $e\in B$ is a primitive idempotent. Parshall \& Scott \cite[Section 5]{ParshallScott} show that this is the reason why $\cd^b(A_2)$ is not the derived category of a quasi-hereditary algebra. \end{rem}

\begin{lem}\label{L:Notquasi}
The algebra $A_2$ does not admit a quasi-hereditary structure.
\end{lem}
\begin{proof}
We show that $A:=A_2$ does not admit a heredity ideal $AeA$, where $e=e_i \in A$ is a {\em primitive} idempotent. Therefore, $A$ is not quasi-hereditary by Remark \ref{R:QH} (a).
Namely, by Remark \ref{R:QH} (b), the existence of such an ideal would imply that the canonical functor $\cd^b(A/AeA) \hookrightarrow \cd^b(A)$ is fully faithful. But this is impossible since there are relations in $A$ between any pair of vertices, yielding non-trivial $\Ext^2_A$- spaces between the corresponding simple $A$-modules. But $\Ext^2_{A/AeA}(-, -)=0$ for any primitive idempotent $e \in A$. Therefore, the canonical functor cannot be full. 
 \end{proof}

Summing up, we obtain a negative answer to Liu \& Yang's Question \ref{Q:LY}.

\begin{prop}\label{P:NegAnswer}
The derived category $\cd^b(A_2)$ of the gentle algebra $A_2$ admits a composition series by derived categories $\cd^b(k)$ but $A_2$ is not derived equivalent to a quasi-hereditary algebra.
\end{prop}
\begin{proof}
The first statement is Lemma \ref{L:Strat}. The second statement is a combination of Lemma \ref{L:Notquasi} and Corollary \ref{C:derivedisolated}. 
\end{proof}

\begin{rem}
For finite dimensional quiver algebras $A$ with two vertices Liu \& Yang show that Question \ref{Q:LY} has a positive answer. In other words if $\cd^b(A)$ admits a composition series by derived categories $\cd^b(k)$, then $A$ is derived equivalent to a quasi-hereditary algebra. In that sense $A_2$ is a minimal counterexample to the question. Note that $A_1$ admits a quasi-hereditary structure since $Q_1$ is directed. Since quasi-hereditary structures are not stable under derived equivalence (see Dlab \& Ringel's \cite[Example p.~283]{DR}) this cannot be used to distinguish the derived categories of $A_1$ \& $A_2$ directly. This motivates the following question. 

\end{rem}

\begin{Question}
Is it possible to characterise derived categories of quasi-hereditary algebras among algebraic triangulated categories? 
\end{Question}

We refer to \cite[Section 5]{ParshallScott} for a first answer to this question -- unfortunately, we are not able to use this characterisation to obtain an alternative proof of Proposition \ref{P:Notequiv}.

\section{Proof of Proposition \ref{P:Notequiv}}\label{S:Proof}

\subsection*{Outline of the proof}

We give a proof by contradiction consisting of the following two steps:

\begin{itemize}
\item[$(i)$] Assume that there exists a triangle equivalence $\Phi\colon \cd^b(A_1) \cong \cd^b(A_2)$. 

Rickard's derived Morita theory \cite{Rickard} shows that there exists a tilting object $T  \in \cd^b(A_1)$ such that there are isomorphisms of graded algebras 
\begin{align}\label{algiso}
\bigoplus_{i \in \ZZ} \Hom_{\cd^b(A_1)}(T, T[i]) \cong \Hom_{\cd^b(A_1)}(T, T) \cong A_2^{\rm op},
\end{align} 
where $A_2^{\rm op}$ is concentrated in degree $0$. Moreover,  $\Phi(T) \cong A_2$ in $\cd^b(A_2)$ and we can assume that $T \cong T_1 \oplus T_2 \oplus T_3$ with $T_i \in \cd^b(A_1)$ indecomposable.

We show that $T_1$ can be chosen to be the string module $1 \xrightarrow{a} 2 \xrightarrow{b} 3$ and $\Phi(T_1) \cong P_2$.
\item[$(ii)$] Let $T_1$ be as in $(i)$. Using the repetitive algebra of $A_1$, we show that there is no indecomposable object $T'$ in $\cd^b(A_1)$ such that
 \begin{align}\label{twocycle}
\begin{array}{lll}
&\bigoplus_{i \in \ZZ} \Hom_{\cd^b(A_1)}(T_1 \oplus T', (T_1 \oplus T')[i]) \cong \End_{\cd^b(A_1)}(T_1 \oplus T') \\ \\ &\cong k\left(\begin{xy}
\SelectTips{cm}{10}
\xymatrix@R=14pt{
T_1 \ar@/^/[rr]|{\displaystyle \,\, x  \,\, }  && T'  \ar@/^/[ll]|{\displaystyle \,\, y \,\, } 
} 
\end{xy}\right)/(xy, yx)=:C
\end{array}
\end{align}
as graded algebras with $C$ concentrated in degree $0$.

This shows that \eqref{algiso} is impossible. Indeed $\Phi$ is an equivalence with $\Phi(T) \cong A_2$ and $\Phi(T_1) \cong P_2$. Therefore, assuming \eqref{algiso}, we have a chain of (graded) isomorphisms \begin{align*} \bigoplus_{i \in \ZZ} \Hom_{\cd^b(A_1)}(T_1 \oplus T_j, (T_1 \oplus T_j)[i]) &\cong \End_{\cd^b(A_1)}(T_1 \oplus T_j) \\ &\cong \End_{\cd^b(A_2)}(\Phi(T_1) \oplus \Phi(T_j)) \\ & \cong \End_{A_2} (P_2 \oplus P_{j'}) \\ &\cong (e_2 + e_{j'}) A_2^{\rm op} (e_2 + e_{j'}) \end{align*}
with $j \in \{2, 3\}$ and $j' \in \{1, 3\}$. One can check that $(e_2+e_1)A_2^{\rm op}(e_2+e_1)\cong C \cong (e_2+e_3)A_2^{\rm op}(e_2+e_3)$. Contradicting the statement that there's no $T' \in \cd^b(A_1)$ satisfying \eqref{twocycle}. Therefore, there cannot exist a derived equivalence $\Phi$. 
 \end{itemize}

\subsection*{Step $(i)$} 
This follows from Corollary \ref{C:Key}. Indeed $A_2$ has AG-invariant $[2, 4]$ and $P_2=2 \xrightarrow{\alpha_1} 3 \xrightarrow{\beta} 1$ is an indecomposable projective $A_2$-module defined by a non-trivial permitted thread in the sense of Definition \ref{D:thread}. Since the $A_1$-module $1 \xrightarrow{a} 2 \xrightarrow{b} 3$ is also defined by a non-trivial permitted thread, Corollary \ref{C:Key} shows the claim.

 \begin{rem}\label{R:Extension}
This step can be generalised to other algebras appearing in Corollary \ref{C:RedBobMali}.
 
 Indeed, the indecomposable projective $B(p, p)$-modules $P_2, \ldots, P_{p+1}$ are strings given by permitted threads satisfying

\begin{align}
\nu(P_{p+1})=P_{p}, \nu(P_{p})=P_{p-1}, \ldots, \nu(P_3)=P_2.
\end{align}

Generalising Corollary \ref{C:Key}, one can show that a tilting object yielding a derived equivalence with $B(p, p)$ has to contain a corresponding sequence of direct  summands $T_2, \ldots, T_{p+1}$ defined by permitted threads.  

It would be interesting to show an analogue of step $(ii)$. Namely, that there is no indecomposable exceptional object $T'$ such that there are isomorphisms of graded algebras
\begin{align}
\begin{array}{ll}
\bigoplus_{i \in \ZZ} \Hom(T' \oplus T_2 \oplus \ldots \oplus T_{p+1}, (T' \oplus T_2 \oplus \ldots \oplus T_{p+1}) [i]) 
\cong \\ \\ \End(T' \oplus T_2 \oplus \ldots \oplus T_{p+1}) \cong \End_{B(p, p)}(P' \oplus P_2 \oplus \ldots \oplus P_{p+1}) \cong kC_{p+1}/J^2,
\end{array}
 \end{align}
 where $C_{p+1}$ denotes a quiver consisting of a single oriented $(p+1)$-cycle, $J\subseteq kC_{p+1}$ is the two-sided ideal generated by the arrows and $P'$ is isomorphic to $P_1$ or $P_{p+2}$.  
\end{rem}

\subsection*{Step $(ii)$} We use Happel's triangle equivalence $H\colon \cd^b(A) \cong \widehat{A}-\ul{\mod}$  for finite dimensional algebras $A$ of finite global dimension (see \cite{Happelbook}) to translate the claim to a question about (string) modules over the repetitive algebra $\widehat{A_1}$. We refer to \cite{ButlerRingel, CB, SZ, WW} for more details on string module combinatorics. We begin by describing the repetitive algebra $\widehat{A_1}$ following Ringel \cite{Ringel} and Schr{\"o}er \cite{Schroer}. Let

\begin{align}\label{E:Repetitive}
\begin{array}{c}
\begin{tikzpicture}[description/.style={fill=white,inner sep=2pt}]
    \matrix (n) [matrix of math nodes, row sep=1em, ampersand replacement=\&,
                 column sep=2.25em, text height=1.5ex, text depth=0.25ex,
                 inner sep=0pt, nodes={inner xsep=0.3333em, inner
ysep=0.3333em}]
    {  
      \widehat{Q_1}= \& \cdots \& -2 \& -1 \& 0 \& 1 \& 2 \&  \cdots  \\
    };
\path[-> ] (n-1-2) edge [bend left=20] node[ scale=0.75, fill=white] [midway] {$a$} 
(n-1-3);
\path[-> ] (n-1-2) edge [bend right=20] node[ scale=0.75, fill=white] [midway] {$b$} 
(n-1-3);
\path[-> ] (n-1-3) edge [bend left=20] node[ scale=0.75, fill=white] [midway] {$a$} 
(n-1-4);
\path[-> ] (n-1-3) edge [bend right=20] node[ scale=0.75, fill=white] [midway] {$b$} 
(n-1-4);
\path[-> ] (n-1-4) edge [bend left=20] node[ scale=0.75, fill=white] [midway] {$a$} 
(n-1-5);
\path[-> ] (n-1-4) edge [bend right=20] node[ scale=0.75, fill=white] [midway] {$b$} 
(n-1-5);
\path[-> ] (n-1-5) edge [bend left=20] node[ scale=0.75, fill=white] [midway] {$a$} 
(n-1-6);
\path[-> ] (n-1-5) edge [bend right=20] node[ scale=0.75, fill=white] [midway] {$b$} 
(n-1-6);
\path[-> ] (n-1-6) edge [bend left=20] node[ scale=0.75, fill=white] [midway] {$a$} 
(n-1-7);
\path[-> ] (n-1-6) edge [bend right=20] node[ scale=0.75, fill=white] [midway] {$b$} 
(n-1-7);
\path[-> ] (n-1-7) edge [bend left=20] node[ scale=0.75, fill=white] [midway] {$a$} 
(n-1-8);
\path[-> ] (n-1-7) edge [bend right=20] node[ scale=0.75, fill=white] [midway] {$b$} 
(n-1-8);
\end{tikzpicture}
\end{array}
\end{align}
and $\widehat{I}:=(a^2, b^2, aba-bab)$. Then the repetitive algebra of $A_1$ is given by $\widehat{A_1}=k\widehat{Q_1}/\widehat{I}$. In particular, the indecomposable projective $\widehat{A_1}$-modules have the following form:
 \begin{align}\label{E:projectives}
 \begin{array}{c}
\begin{xy}
\SelectTips{cm}{10}
\xymatrix@R=14pt{ & n+1 \ar[r]^b & n+2 \ar[rd]^a \\
n \ar[ru]^a \ar[rd]^b &&& n+3 \\
& n+1 \ar[r]^a & n+2 \ar[ru]^b
} 
\end{xy}
\end{array}
\end{align}

\subsubsection*{Preparations}
\begin{itemize}
\item[(a)] We can choose our notation in such a way that $H(1 \xrightarrow{a} 2 \xrightarrow{b} 3) \cong 1 \xrightarrow{a} 2 \xrightarrow{b} 3$.
\item[(b)]  The shift in the triangulated category $\widehat{A}-\ul{\mod}$ is given by the inverse syzygy functor $\Omega^{-1}$. 
\item[(c)] Let $A$ be a finite dimensional $k$-algebra and $n \in \ZZ_{\geq 1}$ then there are isomorphisms 
\begin{align}
\ul{\Hom}_{\widehat{A}}(\Omega^{n}(M),N) \cong \Ext^n_{\widehat{A}}(M, N) \cong \ul{\Hom}_{\widehat{A}}(M,\Omega^{-n}(N)),
\end{align}
where $M, N \in \widehat{A}-\mod$. This is well-known. One can use \cite[Exercises IV.7.2 \& IV.8.3]{HiltonStammbach} together with the fact that projective and injective $\widehat{A}$-modules coincide.
\item[(d)] Assume that there exists an indecomposable object $T'$ satisfying \eqref{twocycle}. In combination with (b) and (c), we see that 
\begin{align} \label{rigid}
\Ext^1_{\widehat{A_1}}(T_1 \oplus T', T_1 \oplus T')=0,
\end{align}
where by abuse of notation we write $T_1$ for $H(T_1)$ and $T'$ for $H(T')$. In particular, $T'$ is given by a string module as band modules always have self-extensions.

\item[(e)] Crawley-Boevey showed that homomorphism spaces $\Hom(M, N)$ between two (indecomposable) string modules $M$ \& $N$ have bases given by graph maps \cite{CB}. In conjunction with Proposition 3.7 of Schr{\"o}er \& Zimmermann \cite{SZ} and \eqref{rigid}, it follows that $\ul{\End}_{\widehat{A_1}}(T_1 \oplus T')$ is generated by (weakly) one-sided graph maps, i.e. maps between string modules $S_1=E\alpha F$ and $S_2=E \beta^- F'$
\[
\begin{tikzpicture}[description/.style={fill=white,inner sep=2pt}]
    \matrix (n) [matrix of math nodes, row sep=1em,
                 column sep=2.25em, text height=1.5ex, text depth=0.25ex,
                 inner sep=0pt, nodes={inner xsep=0.3333em, inner
ysep=0.3333em}]
    {  
      S_1= \quad &&\quad&  \quad    \\
       &&&& \quad &&\quad& \quad \\ \\ \\
       &&&& \quad &&\quad&  \quad    \\
       S_2=\quad &&\quad& \quad \\
    };
\draw[-, decorate, decoration={snake, amplitude=0.75mm} ] (n-1-1) --  node[ scale=0.75, yshift=4mm] [midway] {$E$} (n-1-4); 
\draw[->] ($(n-1-4) + (-9pt, 0.5pt)$) -- node[ scale=0.75, xshift=3mm, yshift=2mm] [midway] {$\alpha$}   ($(n-2-5) +(+9pt, -0.5pt)$);
\draw[-, decorate, decoration={snake, amplitude=0.75mm} ] (n-2-5) -- node[ scale=0.75, yshift=4mm] [midway] {$F$} 
(n-2-8);

\path[->, dashed] ($(n-1-3)+(-12pt, -5pt)$) edge node[ scale=0.75, xshift=4mm] [midway] {$\mathrm{id}$}  ($(n-6-3)+(-12pt, 15pt)$);
\path[->, dashed] ($(n-2-7)+(-13pt, -5pt)$) edge node[ scale=0.75, xshift=4mm] [midway] {$0$}  ($(n-5-7) +(-13pt, 15pt)$);

\draw[-, decorate, decoration={snake, amplitude=0.75mm} ] (n-6-1) --  node[ scale=0.75, yshift=4mm] [midway] {$E$} (n-6-4); 
\draw[->] ($(n-5-5)  + (9pt, 0.5pt)$)  -- node[ scale=0.75, xshift=3mm, yshift=5mm] [midway] {$\beta$}  ($(n-6-4) + (-9pt, -0.5pt)$);
\draw[-, decorate, decoration={snake, amplitude=0.75mm} ] (n-5-5) -- node[ scale=0.75, yshift=4mm] [midway] {$F'$} 
(n-5-8);

%


\end{tikzpicture}
\] 
defined as identity map from the factorstring $E$ of $S_1$ to the substring $E$ of $S_2$ and as zero everywhere else. Here $\alpha$ and $\beta$ are arrows in the quiver $\widehat{Q_1}$, $F$ is a substring of $S_1$ and $F'$ is a factorstring of $S_2$. 

\item[(f)] Using Happel's equivalence $H$, \eqref{twocycle} translates to the following statement. The stable endomorphism algebra $\ul{\End}_{\widehat{A_1}}(T_1 \oplus T')$ of $T_1 \oplus T'$ is isomorphic to $C=k(\begin{xy}
\SelectTips{cm}{10}
\xymatrix@R=8pt{
T_1 \ar@/^/[rr]|{\displaystyle \,\, \ul{x}  \,\, }  && T'  \ar@/^/[ll]|{\displaystyle \,\, \ul{y} \,\, } 
} 
\end{xy})/(\ul{x}\ul{y}, \ul{y}\ul{x})$. We show that this leads to a contradiction, see the last paragraph `Final step' below. By part (e), every stable morphism from $T_1$ to $T'$ is a linear combination of equivalence classes of (weakly) one-sided graph maps, which are given by a factorstring $E_f$ in $T_1$ and a corresponding substring $E_f$ in $T'$. In turn, stable morphisms from $T'$ to $T_1$ are generated by equivalence classes of maps given by substrings $E_s$ in $T_1$ and corresponding factorstrings $E_s$ in $T'$. 
Since $\ul{\Hom}_{\widehat{A_1}}(T_1,T') \cong k \ul{x}$ and $\ul{\Hom}_{\widehat{A_1}}(T',T_1) \cong k \ul{y}$ are one dimensional, 
there are (weakly) one sided graph maps $x \in \Hom_{\widehat{A_1}}(T_1,T')$ and $y \in \Hom_{\widehat{A_1}}(T',T_1)$ representing $\ul{x} \in \ul{\Hom}_{\widehat{A_1}}(T_1,T')$ and $\ul{y} \in \ul{\Hom}_{\widehat{A_1}}(T',T_1) $. 

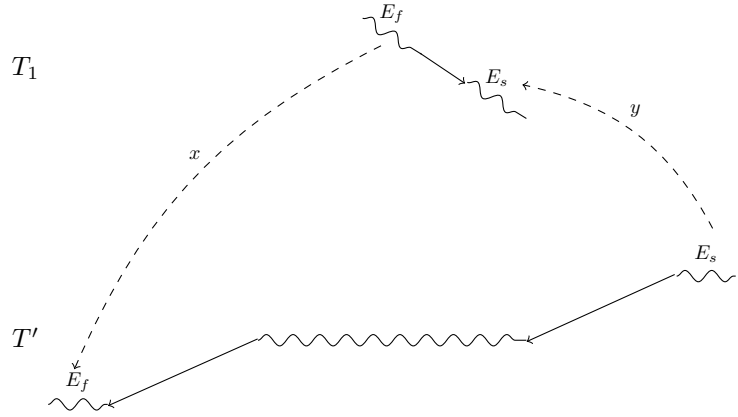
\[
\begin{tikzpicture}[description/.style={fill=white,inner sep=2pt}]
    \matrix (n) [matrix of math nodes, row sep=1em,
                 column sep=2.25em, text height=1.5ex, text depth=0.25ex,
                 inner sep=0pt, nodes={inner xsep=0.3333em, inner
ysep=0.3333em}]
    {  
       &&&\quad     \\
       T_1 &&&\quad &\quad & \quad &\quad& \quad \\ 
       &&&&&\quad&  \quad    \\ \\ \\ \\
       &&&&&&\quad & \quad \\
       T' && \quad &&\quad&  \quad    \\
       \quad &\quad& \quad \\
       };
\draw[-, decorate, decoration={snake, amplitude=0.75mm} ] (n-1-4) --  node[ scale=0.75, yshift=4mm] [midway] {$E_f$} (n-2-5); 
\draw[->] ($(n-1-4) + (30pt, -18.5pt)$) -- ($(n-3-6) +  (-31.5pt, 18.5pt)$);

\draw[-, decorate, decoration={snake, amplitude=0.75mm} ] (n-2-5) -- node[ scale=0.75, yshift=4mm] [midway] {$E_s$} 
(n-3-6);

\path[->, dashed] ($(n-2-4) + (+15.4pt, 8.5pt)$)  edge[bend right=20] node[ scale=0.75, yshift=4mm] [midway] {$x$}  ($(n-9-2.north) +(-21pt, 6pt)$);
\path[->, dashed] ($(n-7-7) + (22pt, 18pt)$) edge[bend right=25] node[scale=0.75, yshift=4mm] [midway] {$y$}   ($(n-3-6) + (-10pt, 18pt)$);

\draw[-, decorate, decoration={snake, amplitude=0.75mm} ] (n-7-7) -- node[ scale=0.75, yshift=4mm] [midway] {$E_s$} (n-7-8); 

\draw[-, decorate, decoration={snake, amplitude=0.75mm} ] (n-9-1) -- node[ scale=0.75, yshift=4mm] [midway] {$E_f$} (n-9-2); 
\path[->] ($(n-8-3) + (+8.4pt, 0.5pt)$) edge   ($(n-9-2)+ (-8.4pt, -0.5pt)$);
\draw[-, decorate, decoration={snake, amplitude=0.75mm} ] (n-8-3) --
(n-8-6);

\path[->] ($(n-7-7) + (+7.8pt, 0.6pt)$) edge   ($(n-8-6) + (-8.4pt, -0.5pt)$);

%


\end{tikzpicture}
\] 
If the substring $E_f$ and the factorstring $E_s$ overlap inside $T'$, then $0 \neq yx \in \End_{\widehat{A_1}}(T_1)$. Since this space is one dimensional, we have $yx= \lambda \cdot \id_{T_1}$ for some $\lambda \in k^*$. By our assumption $\ul{yx}=\ul{y}\ul{x}=0$, so the identity endomorphism of $T_1$ factors over a projective injective $\widehat{A_1}$-module. This would imply that $T_1 \cong 0$ in the stable module category. Contradiction. So $E_f$ and $E_s$ must not overlap in $T'$. This is indicated in the picture above. 

\item[(g)] We claim that choosing $E_f=T_1=E_s$ implies $\ul{x}\ul{y} \neq 0$, which contradicts our assumption. 

We assume without loss of generality that $x$ and $y$ are given as in the picture above (in particular, $x$ is left-sided and $y$ is right-sided in the terminology of \cite{SZ}.
Since $T'$ is a string module we may write $T'=M(S)$ for a string $S$. Let $S^-$ be the inverse string of $S$, see e.g. \cite[p.4]{SZ}. Then there is a canonical isomorphism $\gamma \colon M(S^-) \to M(S)$. Now $ y \gamma$ and $x$ are both left-sided graph maps in the terminology of \cite{SZ}. It follows from \cite[Lemma 3.3]{SZ} that $\ul{x}\ul{y} \cdot \ul{\gamma}=\ul{x}\ul{y\gamma} \neq 0$.
(the lemma is applicable since $T_1$ is not simple, see Zimmermann's correction \cite{Zimmermann} based on Zhou's thesis \cite[preliminary chapter]{Zhou}). In particular, $\ul{x}\ul{y} \neq 0$ as claimed.

\item[(h)] Without loss of generality, we may assume $E_f=1$ or $E_f=1 \xrightarrow{a} 2$. 

To see this, we observe that the only other choice for  $E_f$ is $T_1$. If $E_f=T_1$, then $E_s \neq T_1$ by (g). Next, $k$-duality $D=\Hom_k(-, k)$ defines an exact anti-autoequivalence of $\widehat{A_1}-\ul{\mod}$. Since $C \cong C^{\rm op}$, the pair of modules $T_1, T'$ satisfies \eqref{twocycle} if and only if the pair $D(T_1), D(T')$ does. Under this duality $D(E_s) \neq D(T_1)$ is a factor string in $D(T_1)$. This proves the claim.
\end{itemize}

\bigskip

The following construction is well-known, see e.g.~\cite[Lemma 3.5]{SZ}. It is used in the final step below to show that every candidate for the string module $T'$ has non-trivial self-extensions, contradicting \eqref{rigid}. 
\begin{lem}\label{L:exact}
Let $S_1$ be the following string
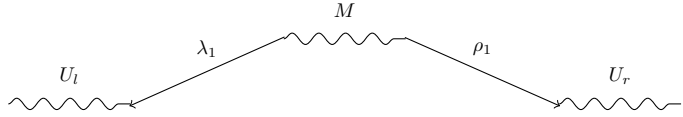
\begin{align*}
\begin{tikzpicture}[description/.style={fill=white,inner sep=2pt}]
    \matrix (n) [matrix of math nodes, ampersand replacement=\&, row sep=1em,
                 column sep=2.25em, text height=1.5ex, text depth=0.25ex,
                 inner sep=0pt, nodes={inner xsep=0.3333em, inner
ysep=0.3333em}]
    {  \&\&\& \quad  \&\& \quad \\
       \quad \&\&\quad \&\&\&\& \quad \&\& \quad \\
       };
\draw[-, decorate, decoration={snake, amplitude=0.75mm} ] (n-2-1) -- node[ scale=0.75, yshift=5mm] [midway] {$U_l$} (n-2-3);
\draw[-, decorate, decoration={snake, amplitude=0.75mm} ] (n-1-4) -- node[ scale=0.75, yshift=5mm] [midway] {$M$} (n-1-6);
\draw[-, decorate, decoration={snake, amplitude=0.75mm} ] (n-2-7) -- node[ scale=0.75, yshift=5mm] [midway] {$U_r$} (n-2-9);
\path[->] ($(n-1-4) + (9pt, 0.7pt)$) edge node[ scale=0.75, yshift=4mm, xshift=0] [midway] {$\lambda_1$}  ($(n-2-3) + (-9pt, -0.8pt)$) ;
\path[->] ($(n-1-6)+ (-9pt, 0.5pt)$) edge node[ scale=0.75, yshift=4mm, xshift=+0mm] [midway] {$\rho_1$}  ($(n-2-7) + (9pt, -0.8pt)$) ;
\end{tikzpicture}
\end{align*} 
and let $S_2$ be given by
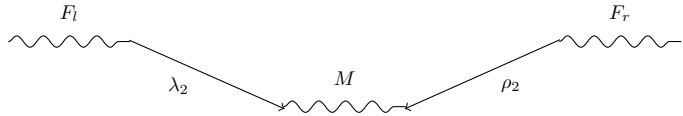
\begin{align*}
\begin{tikzpicture}[description/.style={fill=white,inner sep=2pt}]
    \matrix (n) [matrix of math nodes, ampersand replacement=\&, row sep=1em,
                 column sep=2.25em, text height=1.5ex, text depth=0.25ex,
                 inner sep=0pt, nodes={inner xsep=0.3333em, inner
ysep=0.3333em}]
    {  
    \quad \&\&\quad \&\&\&\& \quad \&\& \quad \\
    \&\&\& \quad  \&\& \quad \\
       };
\draw[-, decorate, decoration={snake, amplitude=0.75mm} ] (n-1-1) -- node[ scale=0.75, yshift=5mm] [midway] {$F_l$} (n-1-3);
\draw[-, decorate, decoration={snake, amplitude=0.75mm} ] (n-2-4) -- node[ scale=0.75, yshift=5mm] [midway] {$M$} (n-2-6);
\draw[-, decorate, decoration={snake, amplitude=0.75mm} ] (n-1-7) -- node[ scale=0.75, yshift=5mm] [midway] {$F_r$} (n-1-9);
\path[->] ($(n-1-3) + (-9pt, 0.5pt)$) edge node[ scale=0.75, yshift=-2mm, xshift=-5mm] [midway] {$\lambda_2$}  ($(n-2-4) + (9pt, -0.8pt)$) ;
\path[->] ($(n-1-7)  + (9pt, 0.7pt)$) edge node[ scale=0.75, yshift=-2mm, xshift=+5mm] [midway] {$\rho_2$}  ($(n-2-6) + (-9pt, -0.8pt)$);
\end{tikzpicture}
\end{align*} 
Here $U_l, U_r, F_l, F_r$ are allowed to be empty -- for example, \eqref{E:Flempty} shows $S_2$ with empty $F_l$.
There exists a  short exact sequence

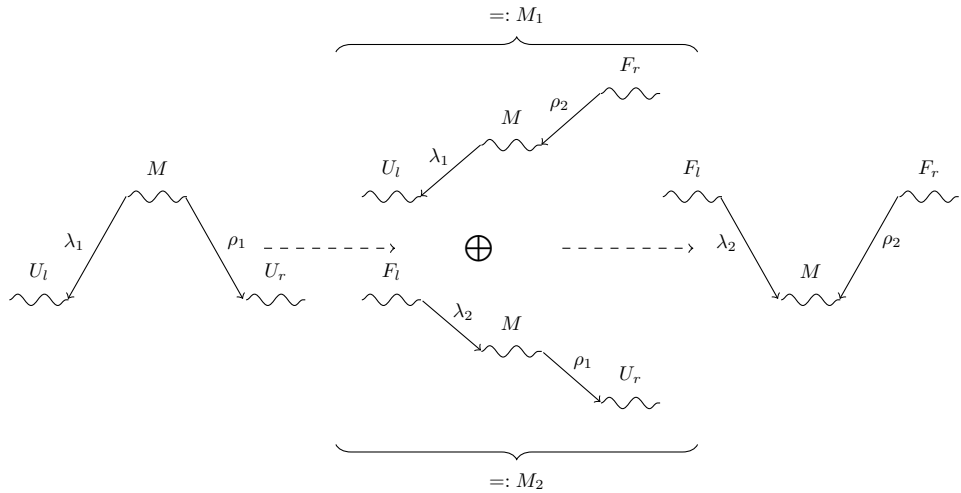
\begin{align*}
\begin{tikzpicture}[description/.style={fill=white,inner sep=2pt}]
    \matrix (n) [matrix of math nodes, ampersand replacement=\&, row sep=0.5em,
                 column sep=0.5em, text height=1.5ex, text depth=0.25ex,
                 inner sep=0pt, nodes={inner xsep=0.3333em, inner
ysep=0.3333em}]
    {                                          \&\&\&\&\&\&\quad \&  \&\&\& \quad \& \quad\\
                                                                   \&\&\&\&\&\&\&\& \quad \&  \quad                                   \\ 
        \&\& \quad  \& \quad               \&\&\& \quad \& \quad                                             \&\&\& \&  \quad \&\quad \&\&\& \quad \& \quad\\
                                                                                \&\&\&\& \quad \&\&\& \quad \& \bigoplus \&\quad \&\&\& \quad \\ 
        \quad \&\quad \&\&\& \quad \& \quad             \&  \quad \& \quad                                        \&\&\& \& \&\& \quad  \& \quad \\
                                                                             \&\&\&\&\&\&\&\& \quad \& \quad                                                          \\
                                                                                     \&\&\&\&\&\& \quad \&\&\&\& \quad \& \quad  \\
       };
\draw[->, dashed] (n-4-5) -- 
(n-4-8);
\draw[->, dashed] (n-4-10) -- 
(n-4-13);
\draw[-, decorate, decoration={snake, amplitude=0.75mm} ] ($(n-5-1) + (2pt, 0)$) -- node[ scale=0.75, yshift=5mm] [midway] {$U_l$} ($(n-5-2) + (2pt, 0)$);
\draw[-, decorate, decoration={snake, amplitude=0.75mm} ] ($(n-3-3) + (2pt, 0)$) -- node[ scale=0.75, yshift=5mm] [midway] {$M$} ($(n-3-4)+ (2pt, 0)$);
\draw[-, decorate, decoration={snake, amplitude=0.75mm} ] ($(n-5-5)+ (2pt, 0)$) -- node[ scale=0.75, yshift=5mm] [midway] {$U_r$} ($(n-5-6) + (2pt, 0)$);
\draw[->] ($(n-3-3)  + (1.4pt, 0.1pt)$) -- node[ scale=0.75, xshift=-4mm, yshift=1mm] [midway] {$\lambda_1$}  ($(n-5-2)+ (1.8pt, 0.4pt)$);
\draw[->] ($(n-3-4) + (1.6pt, -0.4pt)$)  -- node[ scale=0.75, xshift=4mm, yshift=1mm] [midway] {$\rho_1$}   ($(n-5-5)+ (0.8pt, 0.4pt)$);
\draw[decoration={brace, amplitude=2mm}, decorate] ($(n-1-7.west) +(0, 5.5mm)$) -- node[ scale=0.75, yshift=6.5mm] [midway] {$=:M_1$} ($(n-1-11.east)+(1, 5.5mm)$);
\draw[-, decorate, decoration={snake, amplitude=0.75mm} ] ($(n-3-7) + (1pt, 0)$)  -- node[ scale=0.75, yshift=5mm] [midway] {$U_l$} ($(n-3-8) + (1pt, 0)$) ;
\draw[-, decorate, decoration={snake, amplitude=0.75mm} ] ($(n-2-9) + (1pt, 0)$)  -- node[ scale=0.75, yshift=5mm] [midway] {$M$} ($(n-2-10) + (1pt, 0)$) ;
\draw[-, decorate, decoration={snake, amplitude=0.75mm} ] ($(n-1-11) + (1pt, 0)$)  -- node[ scale=0.75, yshift=5mm] [midway] {$F_r$} ($(n-1-12) + (1pt, 0)$) ;
\draw[->] ($(n-2-9) + (0.7pt, 0.05pt)$)-- node[ scale=0.75, xshift=-2mm, yshift=2.5mm] [midway] {$\lambda_1$} ($ (n-3-8) + (0.9pt, 0.2pt)$);
\draw[->] ($(n-1-11) + (0.7pt, 0.05pt)$) -- node[ scale=0.75, xshift=-2mm, yshift=2.5mm] [midway] {$\rho_2$}  ($(n-2-10) + (0.9pt, 0.2pt)$);
\draw[decoration={brace, amplitude=2mm}, decorate]  ($(n-7-11.east)+(1, -5.5mm)$) -- node[ scale=0.75, yshift=-6.5mm] [midway] {$=:M_2$}  ($(n-7-7.west) +(0, -5.5mm)$);
\draw[-, decorate, decoration={snake, amplitude=0.75mm} ] ($(n-5-7) + (1pt, 0)$) -- node[ scale=0.75, yshift=5mm] [midway] {$F_l$} ($(n-5-8) + (1pt, 0)$) ;
\draw[-, decorate, decoration={snake, amplitude=0.75mm} ] ($(n-6-9) + (1pt, 0)$)  -- node[ scale=0.75, yshift=5mm] [midway] {$M$} ($(n-6-10)+ (1pt, 0)$) ;
\draw[-, decorate, decoration={snake, amplitude=0.75mm} ] ($(n-7-11) + (1pt, 0)$)  -- node[ scale=0.75, yshift=5mm] [midway] {$U_r$} ($(n-7-12) + (1pt, 0)$) ;
\draw[->] ($(n-5-8) + (1.6pt, -0.4pt)$)-- node[ scale=0.75, xshift=2mm, yshift=2.5mm] [midway] {$\lambda_2$}   ($(n-6-9) + (0.8pt, 0.4pt)$);
\draw[->] ($(n-6-10) + (1.6pt, -0.4pt)$) -- node[ scale=0.75, xshift=2mm, yshift=2.5mm] [midway] {$\rho_1$}  ($(n-7-11)  + (0.8pt, 0.4pt)$);
\draw[-, decorate, decoration={snake, amplitude=0.75mm} ] ($(n-3-12) + (2pt, 0)$)-- node[ scale=0.75, yshift=5mm] [midway] {$F_l$} ($(n-3-13) + (2pt, 0)$);
\draw[-, decorate, decoration={snake, amplitude=0.75mm} ] ($(n-5-14) + (2pt, 0)$) -- node[ scale=0.75, yshift=5mm] [midway] {$M$} ($(n-5-15) + (2pt, 0)$);
\draw[-, decorate, decoration={snake, amplitude=0.75mm} ] ($(n-3-16) + (2pt, 0)$) -- node[ scale=0.75, yshift=5mm] [midway] {$F_r$} ($(n-3-17) + (2pt, 0)$);
\draw[->] ($(n-3-13) + (1.6pt, -0.4pt)$)  -- node[ scale=0.75, xshift=-4mm, yshift=1mm] [midway] {$\lambda_2$}   ($(n-5-14) + (0.8pt, 0.4pt)$);
\draw[->] ($(n-3-16) + (1.4pt, 0.1pt)$) -- node[ scale=0.75, xshift=4mm, yshift=1mm] [midway]  {$\rho_2$}  ($(n-5-15) + (1.8pt, 0.4pt)$);
\end{tikzpicture}
\end{align*}

provided $M_1$ and $M_2$ are well-defined strings. This sequence does not split provided at least one of  $U_l$ and $F_l$ is non-empty and at least one of $U_r$ and $F_l$ is non-empty. 

\end{lem}

\subsubsection*{Final step} \label{Splitting the string}

Assume that there exists $T'$ satisfying \eqref{twocycle}. Then the 
 preliminary step (f) above implies that there is a pair $(E_s, E_f)$ consisting of a substring $E_s$ and 
a factorstring $E_f$ in $T_1$. Moreover, $E_s$ appears as a factorstring  and $E_f$ as a substring in $T'$ which do not overlap. By step (h), we may assume $E_f= 1$ or $1 \xrightarrow{a} 2$. We treat the case $E_f=1$ first and show that the other case follows from this.

 We prove below that we can assume that $T'$ has the following form (with $E_f=1$ appearing as a substring on the left):

\begin{align}\tag{S0} \label{E:S0}
\begin{array}{c}
\begin{tikzpicture}[description/.style={fill=white,inner sep=2pt}]
    \matrix (n) [matrix of math nodes, ampersand replacement=\&, row sep=1em,
                 column sep=2.25em, text height=1.5ex, text depth=0.25ex,
                 inner sep=0pt, nodes={inner xsep=0.3333em, inner
ysep=0.3333em}]
    {  
         \& 0    \\
        1 \&\& 1 \&\&\&\&\& \quad \\
       };
\draw[-, decorate, decoration={snake, amplitude=0.75mm} ] (n-2-3) -- node[ scale=0.75, yshift=5mm] [midway] {$S$} (n-2-8);
\draw[->] (n-1-2) edge node[ scale=0.75, yshift=2mm, xshift=-5mm] [midway] {$a$}  (n-2-1);
\draw[->] (n-1-2) edge node[ scale=0.75, yshift=2mm, xshift=+5mm] [midway] {$b$}  (n-2-3);
\end{tikzpicture}
\end{array}
\end{align} 

We note that we have to start with the arrow $a$. Otherwise, the morphism $x \colon T_1 \ra T'$ would factor over the indecomposable projective $\widehat{A_1}$-module $P_0$ (see \eqref{E:projectives}), so $\ul{x}=0 \in \ul{\Hom}_{\widehat{A_1}}(T_1, T')$. Moreover, the string defining $T'$ has to reach vertex $3$ at some point (otherwise we don't get a non-zero morphism to $T_1$). This forces $S$ to have the following shape 

\[
\begin{tikzpicture}[description/.style={fill=white,inner sep=2pt}]
    \matrix (n) [matrix of math nodes, row sep=1em,
                 column sep=2.25em, text height=1.5ex, text depth=0.25ex,
                 inner sep=0pt, nodes={inner xsep=0.3333em, inner
ysep=0.3333em}]
    {  
         \quad &&1&& 1    \\
                   &&&2&& 2 && \quad \\
       };
\draw[-, decorate, decoration={snake, amplitude=0.75mm} ] (n-2-6) --  (n-2-8);
\draw[-, decorate, decoration={snake, amplitude=0.75mm} ] (n-1-1) --  (n-1-3);

\draw[->] (n-1-3) edge node[ scale=0.75, yshift=-2mm, xshift=-5mm] [midway] {$a$}  (n-2-4);
\draw[->] (n-1-5) edge node[ scale=0.75, yshift=-2mm, xshift=+4mm] [midway] {$b$}  (n-2-4);
\draw[->] (n-1-5) edge node[ scale=0.75, yshift=-2mm, xshift=-4mm] [midway] {$a$}  (n-2-6);

\end{tikzpicture}
\] 

(indeed otherwise we can never reach a vertex smaller than $2$ since the longest paths without relations in $(\widehat{Q_1}, \widehat{I})$ have length $2$). 

We apply Lemma \ref{L:exact} to obtain a non-trivial self-extension of the string $T'$ contradicting \eqref{rigid}. In order to do this, we write $T'$ in two different ways:

\[
\begin{tikzpicture}[description/.style={fill=white,inner sep=2pt}]
    \matrix (n) [matrix of math nodes, row sep=1em,
                 column sep=1.7em, text height=1.5ex, text depth=0.25ex,
                 inner sep=0pt, nodes={inner xsep=0.3333em, inner
ysep=0.3333em}]
    {  && 0 & \quad&&&&&&& \quad   \\
       S_1:=T'= &1 && 1 &&1&& 1   \\
                       &\quad&& &&&2&& 2 && \quad \\
       };
\draw[-, decorate, decoration={snake, amplitude=0.75mm} ] (n-3-9) --  (n-3-11);
\draw[-, decorate, decoration={snake, amplitude=0.75mm} ] (n-2-4) --  (n-2-6);

\draw[->] (n-2-6) edge node[ scale=0.75, yshift=-2mm, xshift=-5mm] [midway] {$a$}  (n-3-7);
\draw[->] (n-2-8) edge node[ scale=0.75, yshift=-2mm, xshift=+4mm] [midway] {$b$}  (n-3-7);
\draw[->] (n-2-8) edge node[ scale=0.75, yshift=-2mm, xshift=-4mm] [midway] {$a$}  (n-3-9);

\draw[->] (n-1-3) edge node[ scale=0.75, yshift=2mm, xshift=-5mm] [midway] {$a$}  (n-2-2);
\draw[->] (n-1-3) edge node[ scale=0.75, yshift=2mm, xshift=+5mm] [midway] {$b$}  (n-2-4);

\draw[decoration={brace, amplitude=1.5mm}, decorate]  ($(n-2-8.west) +(-0.2mm, 3.5mm)$) -- node[ scale=0.75, yshift=6.5mm] [midway] {$=:M$}   ($(n-2-8.east)+(0.2mm, 3.5mm)$) ;

\draw[decoration={brace, amplitude=1.5mm}, decorate]  ($(n-1-4.west) +(+6.2mm, 3.5mm)$) -- node[ scale=0.75, yshift=6.5mm] [midway] {$:=S$}   ($(n-1-11.east)+(0.2mm, 3.5mm)$) ;

\draw[decoration={brace, amplitude=1.5mm}, decorate]   ($(n-3-7.east)+(0.2mm, -5.5mm)$) -- node[ scale=0.75, yshift=-6.5mm] [midway] {$=:U_l$}     ($(n-3-2.west) +(-0.2mm, -5.5mm)$);

\draw[decoration={brace, amplitude=1.5mm}, decorate]   ($(n-3-11.east)+(-3.5mm, -5.5mm)$) -- node[ scale=0.75, yshift=-6.5mm] [midway] {$=:U_r$}     ($(n-3-9.west) +(-0.2mm, -5.5mm)$);
 
\end{tikzpicture}
\]

\begin{align}\label{E:Flempty}
\begin{split}
\begin{tikzpicture}[description/.style={fill=white,inner sep=2pt}]
    \matrix (n) [matrix of math nodes, ampersand replacement=\&, row sep=1em,
                 column sep=2.25em, text height=1.5ex, text depth=0.25ex,
                 inner sep=0pt, nodes={inner xsep=0.3333em, inner
ysep=0.3333em}]
    {  
              \&\& 0 \&\&\&\& \&\& \quad    \\
       S_2:=T'= \& 1 \&\& 1 \&\&\&\&\& \quad \\
       };
\draw[-, decorate, decoration={snake, amplitude=0.75mm} ] (n-2-4) -- node[ scale=0.75, yshift=5mm] [midway] {$S$} (n-2-9);
\draw[->] (n-1-3) edge node[ scale=0.75, yshift=2mm, xshift=-5mm] [midway] {$a$}  (n-2-2);
\draw[->] (n-1-3) edge node[ scale=0.75, yshift=2mm, xshift=+5mm] [midway] {$b$}  (n-2-4);
\draw[decoration={brace, amplitude=1.5mm}, decorate]  ($(n-2-2.east)+(0.2mm, -3.5mm)$) -- node[ scale=0.75, yshift=-6.5mm] [midway] {$=:M$}  ($(n-2-2.west) +(-0.2mm, -3.5mm)$) ;
\draw[decoration={brace, amplitude=1.5mm}, decorate]  ($(n-1-3.west) +(-0.2mm, 5.5mm)$) -- node[ scale=0.75, yshift=6.5mm] [midway] {$=:F_r$}   ($(n-1-9.east)+(0.2mm, 5.5mm)$)  ;
 \end{tikzpicture}
 \end{split}
\end{align} 

Since the string modules

\[
\begin{tikzpicture}[description/.style={fill=white,inner sep=2pt}]
    \matrix (n) [matrix of math nodes, row sep=1em,
                 column sep=1.7em, text height=1.5ex, text depth=0.25ex,
                 inner sep=0pt, nodes={inner xsep=0.3333em, inner
ysep=0.3333em}]
    {  && 0 &&&&& \quad &                 0  \\
       M_1= &1 && 1 &&1&& 1  && 1 && \quad  \\
                &\quad&& &&&2 & & \quad &&& \quad \\
       };
\draw[-, decorate, decoration={snake, amplitude=0.75mm} ]   (n-2-10) -- node[ scale=0.75, yshift=5mm] [midway]  {$S$}  (n-2-12);
\draw[-, decorate, decoration={snake, amplitude=0.75mm} ] (n-2-4) --  (n-2-6);

\draw[->] (n-2-6) edge node[ scale=0.75, yshift=-2mm, xshift=-5mm] [midway] {$a$}  (n-3-7);
\draw[->] (n-2-8) edge node[ scale=0.75, yshift=-2mm, xshift=+4mm] [midway] {$b$}  (n-3-7);
\draw[->] (n-1-9) edge node[ scale=0.75, yshift=-2mm, xshift=+4mm] [midway] {$a$}  (n-2-8);
\draw[->] (n-1-9) edge node[ scale=0.75, yshift=-2mm, xshift=-4mm] [midway] {$b$}  (n-2-10);

\draw[->] (n-1-3) edge node[ scale=0.75, yshift=2mm, xshift=-5mm] [midway] {$a$}  (n-2-2);
\draw[->] (n-1-3) edge node[ scale=0.75, yshift=2mm, xshift=+5mm] [midway] {$b$}  (n-2-4);

\draw[decoration={brace, amplitude=1.5mm}, decorate]  ($(n-1-8.west) +(-0.2mm, 3.5mm)$) -- node[ scale=0.75, yshift=6.5mm] [midway] {$=:M$}   ($(n-1-8.east)+(0.2mm, 3.5mm)$) ;

\draw[decoration={brace, amplitude=1.5mm}, decorate]   ($(n-3-7.east)+(0.2mm, -5.5mm)$) -- node[ scale=0.75, yshift=-6.5mm] [midway] {$=:U_l$}     ($(n-3-2.west) +(-0.2mm, -5.5mm)$);

\draw[decoration={brace, amplitude=1.5mm}, decorate]   ($(n-3-12.east)+(-3.5mm, -5.5mm)$) -- node[ scale=0.75, yshift=-6.5mm] [midway] {$=:F_r$}     ($(n-3-9.west) +(-0.2mm, -5.5mm)$);
 
\end{tikzpicture}
\] 

and 

\[
\begin{tikzpicture}[description/.style={fill=white,inner sep=2pt}]
    \matrix (n) [matrix of math nodes, row sep=1em,
                 column sep=1.7em, text height=1.5ex, text depth=0.25ex,
                 inner sep=0pt, nodes={inner xsep=0.3333em, inner
ysep=0.3333em}]
    {  M_2= & 1    \\
                 && 2 && \quad \\
       };
\draw[-, decorate, decoration={snake, amplitude=0.75mm} ] (n-2-3) --  (n-2-5);

\draw[->] (n-1-2) edge node[ scale=0.75, yshift=-2mm, xshift=-4mm] [midway] {$a$}  (n-2-3);

\draw[decoration={brace, amplitude=1.5mm}, decorate]  ($(n-1-2.west) +(-0.2mm, 3.5mm)$) -- node[ scale=0.75, yshift=6.5mm] [midway] {$=:M$}   ($(n-1-2.east)+(0.2mm, 3.5mm)$) ;

\draw[decoration={brace, amplitude=1.5mm}, decorate]   ($(n-2-5.east)+(-3.5mm, -5.5mm)$) -- node[ scale=0.75, yshift=-6.5mm] [midway] {$=:U_r$}     ($(n-2-3.west) +(-0.2mm, -5.5mm)$);
 
\end{tikzpicture}
\] 
exist, Lemma \ref{L:exact} shows $\Ext^1_{\widehat{A_1}}(T', T') \neq 0$ as desired.

It remains to show that we can assume that $T'$ starts as indicated in the picture \eqref{E:S0} above. Since we want $1$ to be a substring there are the following two other possibilities -- where $S_1$ and $S_2$ have to be non-trivial (indeed, otherwise there is no non-zero morphism $T' \to T_1$).

\begin{align}\tag{S1}
\begin{array}{c}
\begin{tikzpicture}[description/.style={fill=white,inner sep=2pt}]
    \matrix (n) [matrix of math nodes, ampersand replacement=\&, row sep=1em,
                 column sep=2.25em, text height=1.5ex, text depth=0.25ex,
                 inner sep=0pt, nodes={inner xsep=0.3333em, inner
ysep=0.3333em}]
    {    \&\& -1 \\
         \&0 \&\& 0 \&\&0 \&\&\&\&\& \quad    \\
        1 \&\&\&\& 1  \\
       };
\draw[-, decorate, decoration={snake, amplitude=0.75mm} ] (n-2-6) -- node[ scale=0.75, yshift=4.5mm] [midway] {$S_1$} (n-2-11);
\draw[->] (n-1-3) edge node[ scale=0.75, yshift=4mm, xshift=-3.7mm] [midway] {$b$}  (n-2-2);
\draw[->] (n-1-3) edge node[ scale=0.75, yshift=4mm, xshift=+3.7mm] [midway] {$a$}  (n-2-4);
\draw[->] (n-2-2) edge node[ scale=0.75, yshift=4mm, xshift=-3.7mm] [midway] {$a$}  (n-3-1); 
\draw[->] (n-2-4) edge node[ scale=0.75, yshift=4mm, xshift=+3.7mm] [midway] {$b$}  (n-3-5); 
\draw[->] (n-2-6) edge node[ scale=0.75, yshift=4mm, xshift=-3.7mm] [midway] {$a$}  (n-3-5);
\end{tikzpicture}
\end{array}
\end{align}

\begin{align}\tag{S2}
\begin{array}{c}
\begin{tikzpicture}[description/.style={fill=white,inner sep=2pt}]
    \matrix (n) [matrix of math nodes, ampersand replacement=\&, row sep=1em,
                 column sep=2.25em, text height=1.5ex, text depth=0.25ex,
                 inner sep=0pt, nodes={inner xsep=0.3333em, inner
ysep=0.3333em}]
    {    \&\& -1 \&\&-1 \&\&\&\&\& \quad \\
         \&0 \&\& 0     \\
        1   \\
       };
\draw[-, decorate, decoration={snake, amplitude=0.75mm} ] (n-1-5) -- node[ scale=0.75, yshift=4.5mm] [midway] {$S_2$}  (n-1-10);
\draw[->] (n-1-3) edge node[ scale=0.75, yshift=4mm, xshift=-3.7mm] [midway] {$b$}  (n-2-2);
\draw[->] (n-1-3) edge node[ scale=0.75, yshift=4mm, xshift=+3.7mm] [midway] {$a$}  (n-2-4);
\draw[->] (n-2-2) edge node[ scale=0.75, yshift=4mm, xshift=-3.7mm] [midway] {$a$}  (n-3-1); 
\draw[->] (n-1-5) edge node[ scale=0.75, yshift=4mm, xshift=-3.7mm] [midway] {$b$}  (n-2-4); 
\end{tikzpicture}
\end{array}
\end{align}

If $T'$ starts as in (S1), we apply the syzygy functor $\Omega$ to $T_1 \oplus T'$.
 Using \eqref{E:projectives}, we compute

\begin{align*}
\Omega(T_1)=2 \xrightarrow{a} 3 \xrightarrow{b} 4 
\end{align*}

and

\[
\begin{tikzpicture}[description/.style={fill=white,inner sep=2pt}]
    \matrix (n) [matrix of math nodes, row sep=1em,
                 column sep=2.25em, text height=1.5ex, text depth=0.25ex,
                 inner sep=0pt, nodes={inner xsep=0.3333em, inner
ysep=0.3333em}]
    {  
         &&1    \\
        \Omega(T')=& 2 && 2 &&&&& \quad \\
       };
\draw[-, decorate, decoration={snake, amplitude=0.75mm} ] (n-2-4) -- node[ scale=0.75, yshift=5mm] [midway] {$S'_1$} (n-2-9);
\draw[->] (n-1-3) edge node[ scale=0.75, yshift=4mm, xshift=-3.7mm] [midway] {$a$}  (n-2-2);
\draw[->] (n-1-3) edge node[ scale=0.75, yshift=4mm, xshift=+3.7mm] [midway] {$b$}  (n-2-4);
 
\end{tikzpicture}
\] 

This is a shifted version of (S0) and 
we have already seen that this leads to a contradiction.
Since $\Omega$ is an autoequivalence, we deduce that $T'$ cannot be of the form (S1).

In case $T'$ starts as in (S2), we can repeatedly apply the Auslander-Reiten translation $\tau$
(viewed as an autoequivalence of $\widehat{A_1}-\ul{\mod} $)  to both $T'$ and $T_1$. Since $\widehat{A_1}$ is special biserial \cite{Ringel, Schroer}, the action of $\tau$ on strings is well understood (see e.g. \cite[Thm 4.1]{WW}).

\begin{align*}
\tau^n(T_1) = \begin{cases}  (1-n) \xrightarrow{a} (2-n) \xrightarrow{b} (3-n) \qquad \text{    if } n \text{ is even;} \\  
(1-n) \xrightarrow{b} (2-n) \xrightarrow{a} (3-n) \qquad \text{    if } n \text{ is odd.}   \end{cases} 
\end{align*}

If we apply $\tau$ to $T'$, we `remove a hook'  from the left of the string $T'$: 
\[
\begin{tikzpicture}[description/.style={fill=white,inner sep=2pt}]
    \matrix (n) [matrix of math nodes, row sep=1em,
                 column sep=0.7em, text height=1.5ex, text depth=0.25ex,
                 inner sep=0pt, nodes={inner xsep=0.3333em, inner
ysep=0.3333em}]
    {    && m-2  && m-2 && \, && && m-2 & \, \\
         &m-1 && m-1 && \quad  &\quad & & \quad & m-1 && \quad   \\
        m   \\
       };

\draw[->] (n-1-3) edge   (n-2-2);
\draw[->] (n-1-3) edge   (n-2-4);

\draw[->] (n-2-2) edge   (n-3-1); 

\draw[->] (n-1-5) edge   (n-2-4);

\draw[-, decorate, decoration={snake, amplitude=0.75mm} ] (n-1-5) -- node[ scale=0.75, yshift=5mm] [midway] {$S_2$} (n-1-7);

\draw[-> ] (n-2-6) -- node[ scale=0.75, yshift=5mm] [midway] {$\tau(-)$} (n-2-9);

\draw[->] (n-1-11) edge   (n-2-10); 

\draw[-, decorate, decoration={snake, amplitude=0.75mm} ] (n-1-11) -- node[ scale=0.75, yshift=5mm] [midway] {$S'_2$} (n-1-12);

\end{tikzpicture}
\] 
where $S'_2$ is non-zero (indeed otherwise there is no non-zero morphism $\tau(T') \to \tau(T_1)$).
After repeatedly applying $\tau$, we have removed all these hooks from $T'$ and reach a (shifted) version of (S0) or (S1) (indeed otherwise 
$T'$ would have the following form

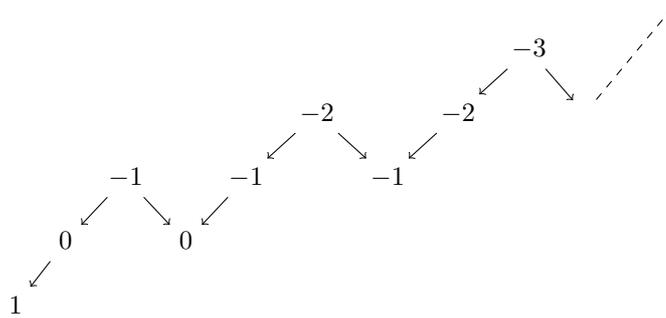
\[
\begin{tikzpicture}[description/.style={fill=white,inner sep=2pt}]
    \matrix (n) [matrix of math nodes, row sep=1em,
                 column sep=0.7em, text height=1.5ex, text depth=0.25ex,
                 inner sep=0pt, nodes={inner xsep=0.3333em, inner
ysep=0.3333em}]
    { 
    &&&&&& &&  &&&             \qquad           \\
    &&&&&   &&&          -3  && \,  \\
    &&&&&       -2 & & -2 && \,  \\  
    && -1 && -1&& -1\\
    & 0 && 0 \\
    1 \\
       };

\draw[->] (n-5-2) edge   (n-6-1);
\draw[->] (n-4-3) edge   (n-5-2);
\draw[->] (n-4-3) edge   (n-5-4);

\draw[->] (n-4-5) edge   (n-5-4); 

\draw[->] (n-3-6) edge   (n-4-5); 
\draw[->] (n-3-6) edge   (n-4-7); 
\draw[->] (n-3-8) edge   (n-4-7);

\draw[->] (n-2-9) edge   (n-3-8);
\draw[->] (n-2-9) edge   (n-3-10);

\draw[-, dashed] (n-3-10) edge   (n-1-12);





%
\end{tikzpicture}
\]

and again there would be no non-zero map from $T'$ to $T_1$. Contradiction), which (as before) leads to a contradiction. 
This completes the proof of the case $E_f=1$. If $E_f=1 \xrightarrow{a} 2$, then there are the following two possibilities for $T'$ (with $E_f$ appearing as a substring on the left).

\begin{align}\tag{S3}
\begin{array}{c}
\begin{tikzpicture}[description/.style={fill=white,inner sep=2pt}]
    \matrix (n) [matrix of math nodes, ampersand replacement=\&, row sep=1em,
                 column sep=2.25em, text height=1.5ex, text depth=0.25ex,
                 inner sep=0pt, nodes={inner xsep=0.3333em, inner
ysep=0.3333em}]
    {    \&\& 0   \\
         \&1 \&\& 1 \&\&\&\&\& \quad    \\
        2   \\
       };
\draw[-, decorate, decoration={snake, amplitude=0.75mm} ] (n-2-4) -- node[ scale=0.75, yshift=4.5mm] [midway] {$S_3$}  (n-2-9);
\draw[->] (n-1-3) edge node[ scale=0.75, yshift=4mm, xshift=-3.7mm] [midway] {$b$}  (n-2-2);
\draw[->] (n-1-3) edge node[ scale=0.75, yshift=4mm, xshift=+3.7mm] [midway] {$a$}  (n-2-4);
\draw[->] (n-2-2) edge node[ scale=0.75, yshift=4mm, xshift=-3.7mm] [midway] {$a$}  (n-3-1); 
\end{tikzpicture}
\end{array}
\end{align}

\begin{align}\tag{S4}
\begin{array}{c}
\begin{tikzpicture}[description/.style={fill=white,inner sep=2pt}]
    \matrix (n) [matrix of math nodes, ampersand replacement=\&, row sep=1em,
                 column sep=2.25em, text height=1.5ex, text depth=0.25ex,
                 inner sep=0pt, nodes={inner xsep=0.3333em, inner
ysep=0.3333em}]
    {    1\&\& 1  \&\&\&\&\& \quad \\
         \&2 \&     \\
       };
\draw[-, decorate, decoration={snake, amplitude=0.75mm} ] (n-1-3) -- node[ scale=0.75, yshift=4.5mm] [midway] {$S_4$}  (n-1-8);
\draw[->] (n-1-3) edge node[ scale=0.75, yshift=-4mm, xshift=3.7mm] [midway] {$b$}  (n-2-2);
\draw[->] (n-1-1) edge node[ scale=0.75, yshift=-4mm, xshift=-3.7mm] [midway] {$a$}  (n-2-2);
\end{tikzpicture}
\end{array}
\end{align}

In case (S3) the morphism $x \colon T_1 \to T'$ defined by $E_f$ factors over the  projective-injective module $P_0$. So $\ul{x}=0$ contradicting our assumptions. In case (S4), we can apply the inverse syzygy functor $\Omega^{-1}$ to $T_1 \oplus T'$.

\begin{align*}
\Omega^{-1}(T_1)=0 \xrightarrow{a} 1 \xrightarrow{b} 2 
\end{align*}

and

\[
\begin{tikzpicture}[description/.style={fill=white,inner sep=2pt}]
    \matrix (n) [matrix of math nodes, row sep=1em,
                 column sep=2.25em, text height=1.5ex, text depth=0.25ex,
                 inner sep=0pt, nodes={inner xsep=0.3333em, inner
ysep=0.3333em}]
    {  
         &&-1    \\
        \Omega^{-1}(T')=& 0 && 0 &&&&& \quad \\
       };
\draw[-, decorate, decoration={snake, amplitude=0.75mm} ] (n-2-4) -- node[ scale=0.75, yshift=5mm] [midway] {$S'_4$} (n-2-9);
\draw[->] (n-1-3) edge node[ scale=0.75, yshift=4mm, xshift=-3.7mm] [midway] {$a$}  (n-2-2);
\draw[->] (n-1-3) edge node[ scale=0.75, yshift=4mm, xshift=+3.7mm] [midway] {$b$}  (n-2-4);
 
\end{tikzpicture}
\] 

This is a shifted version of (S0) and 
we have already seen that this leads to a contradiction. Since $\Omega^{-1}$ is an autoequivalence, we deduce that $T'$ cannot be of the form (S4). This completes the proof.

\medskip
\noindent
\emph{Acknowledgement}. 
My interest in this work was triggered by Igor Burban, who introduced me to the algebra $A_1$ and by Sefi Ladkani who asked me whether Proposition \ref{P:Notequiv} or its negation hold, see \cite{LadkaniNOTE}.
I would like to thank Dong Yang for pointing out the relation to \cite{derivedsimple}, many helpful discussions about my earlier attempts to show Proposition \ref{P:Notequiv} and constant motivation. Thanks to Joe Karmazyn for helping me with another unfortunately unsuccessful approach to prove this proposition. I'm grateful to Jan Schr{\"o}er for introducing me to representation theory, for discussions about \cite{SZ} (in particular, for refering me to \cite{Zimmermann}) and for helpful remarks.

This work was presented at the annual conference of the \emph{DFG-Priority Program in Representation Theory} (Bad Honnef 2015) and the workshop \emph{Silting theory and related topics}Ê (Verona 2015). Moreover, I had fruitful discussions about this article during my stay at the Mittag-Leffler Institute \emph{Research Program on Representation theory} (thanks, in particular, to Claus Ringel for his interest and his enthusiasm!) and the follow-up workshop on the \emph{Interaction of Representation Theory with Geometry and Combinatorics} at HIM Bonn (thanks to Lidia Angeleri H{\"u}gel, Steffen K{\"o}nig, Frederik Marks \& Jorge Vit{\'o}ria) in March 2015. I'm grateful to the organisers of all these events for giving me the opportunity to present and discuss my work and for creating a good working atmosphere.

I thank the anonymous referee for reading the article carefully and for many helpful comments and suggestions, which helped to improve the paper. I also thank Pieter Belmans, Sefi Ladkani and Yuya Mizuno, for  helpful remarks on an earlier version.

\end{document}